\documentclass[12pt]{amsart}
\usepackage{amsmath,amsthm}
\usepackage{geometry}                % See geometry.pdf to learn the layout options. There are lots.
\usepackage{graphicx}
\usepackage{amssymb}
\usepackage{epstopdf}
 \usepackage{xcolor}
\newtheorem{theorem}{Theorem}[section]
\newtheorem{corollary}[theorem]{Corollary}
\newtheorem{lemma}[theorem]{Lemma}
\newtheorem{proposition}[theorem]{Proposition}

\theoremstyle{definition}
\newtheorem{definition}[theorem]{Definition}

\theoremstyle{remark}
\newtheorem{remark}[theorem]{Remark}

\numberwithin{equation}{section}

\begin{document}

\title[Stronger Radial Attraction] 
 {Stronger Radial Attraction: A Generalization of Radial Curvature bounded from above}

\author{James J. Hebda }
\address{Department of Mathematics and Statistics,
Saint Louis University, St. Louis, MO 63103, USA}
\email{james.hebda@slu.edu} 

 %   author two information
 
 \author{Yutaka Ikeda}

\address{8 Spring Time CT, St. Charles, MO 63303, USA }
\email{yutaka.ikeda08@gmail.com}


%\date{}                                           % Activate to display a given date or no date


\maketitle

\begin{abstract}
This paper introduces and investigates a generalization of the notion of a  pointed Riemannian manifold having its radial curvature bounded from above by that of a model surface of revolution.   
\end{abstract}

\section{Preliminaries}

\subsection{Introduction} 
In the series of papers \cite{HI1,HI2, HI3}, we investigated a generalization of the notion of a Riemannian manifold whose radial curvature was bounded from below by that of a model surface of revolution. We called this generalization weaker radial attraction and deduced a number of consequences which included a triangle comparison theorem, a sphere theorem, and a maximal radius theorem. This paper is complementary to that work in that it studies the analogous generalization of the notion of a Riemannian manifold having radial curvature bounded from above by that of a model surface of revolution. We call this generalization stronger radial attraction. Here we prove a triangle comparison theorem and apply it to estimate the convexity radius at the base point in the Riemannian manifold in terms of the convexity radius of the vertex in the model surface.  We also show if a model surface has stronger radial attraction than a given pointed manifold, then the volume and first eigenvalue of  geodesic balls of small radius about the base point in the manifold are bounded from below by those of the geodesic ball of the same radius in the model space of the same dimension.

In contrast to the weaker radial attraction case, where every based geodesic triangle in the manifold has a corresponding triangle in the model which satisfies a convexity property and an angle comparison, this is no longer true in the stronger radial attraction case.  First, there must be restrictions (See Equation (\ref{e:exist})) on the sides of a based triangle for there to exist a comparison triangle in the model with the same side lengths, but even this is not enough to obtain the convexity property and angle comparison.  To ensure this, one needs to assume further that
the distance from the base point to the side opposite the base vertex of the triangle has a 2--sided derivative along the side. (See Theorem \ref{t:TCT} and Corollary \ref{c:angle}.)

We also address the question of whether radial curvature bounded from above implies stronger radial attraction by proving that is true when the manifold is real analytic in Corollary \ref{c:CS}.  This result depends on Lemma \ref{l:rays} which describes a property of short geodesic rays emanating from a nonconjugate cut point. This result, as far as we know, is new to the literature.

Combining the results of this paper about stronger radial attraction with results about weaker radial attraction from \cite{HI1,HI2}, we prove two  triangle pinching theorems (Theorem \ref{t:TPT} and Corollary \ref{c:TPT}).
These two theorems  are  analogous to the Sphere Theorem and the Minimal Diameter Theorem concerning manifolds whose sectional curvatures are  pinched to lie between $\frac 14$ and $1$.

\subsection{Notation} 
All manifolds under discussion are smooth $(C^\infty)$ unless otherwise noted.  The distance between two points $p$ and $q$ in a Riemannian manifold is denoted by $d(p,q)$, and the distance function from a point $p$ in the manifold is denoted by $L_p$, that is, $L_p(q) = d(p,q)$ for all $q$. The cut locus of a point $p$ in a Riemannain manifold is denoted by $C(p)$.  The reader will find more background material in \cite{HI1,HI2}. 

\subsection{Models} 
Let $(\widetilde M, \tilde o)$ be a complete simply connected two--dimensional Riemannian manifold which is rotationally symmetric  about the vertex $\tilde o$.  Then, in geodesic polar coordinates $(r,\theta) \in (0,\ell)\times [0,2\pi]$  around $\tilde o$, the metric takes the form
$$ds^2 = dr^2 + y(r)^2 d\theta^2$$
for some $\ell \in (0,\infty]$.  
When  $\ell <\infty$, $\widetilde M$ is compact and diffeomorphic to the 2--sphere, while if  $\ell =\infty$, $\widetilde M$ is noncompact and diffeomorphic to the Euclidean plane. Since it is assumed that the metric is smooth, $y(r)$ extends to an odd $C^\infty$ function on the real line, satisfying $y(0)=0$, $y^\prime(0)=1$, and $y(r) >0$ on $(0,\ell)$.  Moreover, if $ \ell<\infty$, then
$y(\ell)=0$, $y^\prime(\ell)=-1$, and $y(r)$ is an odd function about $\ell$.  Any such $(\widetilde M,\tilde o)$ will be called a (pointed) model surface. The notation  
$\widetilde M^+$  denotes the set of points in $\widetilde M$ whose polar angle satisfies $0\leq \theta \leq \pi$.

The corresponding model spaces of dimension $n$ will be denoted $(\widetilde M^n, \tilde o^n)$. 
In geodesic spherical coordinates $(r,\theta) \in (0,\ell)\times S^{n-1}$, the metric takes the form
$$ds^2 = dr^2 + y(r)^2d\theta_{n-1}^2$$
where $d\theta_{n-1}^2$ is the standard metric on the unit $(n-1)$--dimensional sphere and $y(r)$ is the same function as in the model surface. Clearly, this metric is rotationally symmetric about $\tilde o^n$ under the action of the orthogonal group $O(n)$.

We need two functions that are associated to the model surface $(\widetilde M,\tilde o)$.

 \begin{definition} %1.1
 For $0 \leq \theta \leq \pi$ and $ 0 \leq \phi \leq \pi$, define the functions
 \begin{eqnarray*}
&D_\theta: (0, \ell)\times [0,\ell) \rightarrow \mathbb{R}^+ \cr
&R_\phi : (0,\ell)\times [0,\infty) \rightarrow \mathbb{R}^+ 
\end{eqnarray*}
as follows:  $D_\theta(r_1,r_2) = L_{\tilde p}( \tilde q)$ where $\tilde p$ has coordinates $(r_1,0)$ and $\tilde q \in \widetilde M^+$ has coordinates $(r_2,\theta)$.  $R_\phi(r_1, t) = L_{\tilde o}( \sigma(t))$
where  $\sigma$ is the unit speed geodesic starting at $\tilde p =  (r_1,0)$ making an angle $\phi$ with the meridian $\theta = 0$. 
 \end{definition}

The next proposition is proved in \cite[Proposition 3.2]{HI1}. See also \cite[Lemma~2.1]{IMS}, \cite[Lemma 7.3.2]{SST} and \cite[Lemma 5.1]{KT} for similar results.
 
 \begin{proposition} \label{p:monotonicity} %1.2
  $D_\theta$ and $R_\phi$ have the following monotonicity properties. 
 \begin{enumerate}
\item  For fixed $(r_1, r_2) \in (0,\ell)\times(0,\ell)$, $D_\theta(r_1,r_2)$ is strictly increasing for $0 \leq \theta \leq \pi$.  
 \item  For fixed $(r_1,t) \in (0,\ell)\times (0,\infty)$, if $t$ is less than the injectivity radius of $\tilde p$, then $R_\phi(r_1,t)$ is strictly decreasing for $ 0 \leq \phi \leq \pi$.
 \end{enumerate}
 \end{proposition}

\section{Stronger Radial Attraction}
 
 Let $M$ be a complete  Riemannian manifold and $o$ a point in $M$.  Let $(\widetilde M, \tilde o)$ be a model surface.
 
 \begin{definition}\label{d:SRA}%2.1
 The model $(\widetilde M,\tilde o)$ has \emph{stronger radial attraction} than the pointed Riemannian manifold $(M,o)$ if whenever $\sigma$ and $\tilde \sigma$ are unit speed geodesics in $M$ and $\widetilde M$ respectively such that $L_o\circ \sigma(0) = L_{\tilde o}\circ \tilde \sigma(0)<\ell$ and $(L_o\circ \sigma)^\prime_+(0) = (L_{\tilde o}\circ \tilde \sigma)^\prime_+(0)$, there exists an $\epsilon >0$ such that $L_o\circ \sigma(t) \geq L_{\tilde o}\circ \tilde\sigma(t)$ for all $0 \leq t < \epsilon$.
If this is the case,  we may also say $(M,o)$ has \emph{weaker radial attraction} than $(\widetilde M, \tilde o)$. (cf. \cite[Definition 4.1]{HI1} and \cite[Definition 3.1]{HI2}.)
\end{definition}

That $L_o\circ\sigma$ has both a left and right hand derivative at every point is a consequence of Lemma \ref{l:derivative} stated below  which is proved in \cite[Lemma 2.1]{HI1}.

\begin{lemma}\label{l:derivative}%2.2
Let $M$ be a complete Riemannian manifold with Riemannian metric $g(-,-)= \langle-,-\rangle$. Suppose that $c : (a,b) \rightarrow M$ is a $C^\infty$ curve parameterized by arclength, and that $p\in M$.   Then for each $s \in (a,b) $ the left and right hand derivatives of $L_p\circ c$ exist and are given by
\begin{eqnarray*} 
(L_p\circ c)_+^\prime(s) &=& \min \left\{ \langle c^\prime(s), \gamma^\prime(l) \rangle :  \gamma \in Geod(p,c(s))\right\} \\
(L_p\circ c)_-^\prime(s) &=& \max \left\{ \langle c^\prime(s), \gamma^\prime(l) \rangle :  \gamma \in Geod(p,c(s))\right\} 
\end{eqnarray*}
where $l = dist(p, c(s))$ and $Geod(p,c(s))$ is the set of all minimizing geodesics joining $p$ to $c(s)$. 
\end{lemma}

\begin{remark}\label{r:<1}%2.3
By Lemma \ref{l:derivative}, $-1 \leq (L_o\circ \sigma)^\prime_+(0)\leq 1$.  If $(L_o\circ \sigma)^\prime_+(0)=\pm 1$, then
$\sigma$ would lie along a minimizing geodesic joining $o$ to $\sigma(0)$, and we would have  $L_o\circ\sigma(t) = \pm t +L_o\circ\sigma(0)$ respectively for all small enough $t$. Thus in Definition \ref{d:SRA} we would also have $L_{\tilde o}\circ\tilde\sigma(t) = \pm t +L_o\circ\sigma(0) = L_o\circ\sigma(t)$ for all small enough $t$.  This shows that to check for stronger radial attraction, one only needs to consider cases with $-1 < (L_o\circ \sigma)^\prime_+(0) < 1$.
\end{remark}

\begin{remark}\label{r:2sided}%2.4
For the triangle comparison theorems proved below in Section \ref{s:tct}  it is important to know when $L_o\circ\sigma$ has a two--sided derivative.  According to Lemma \ref{l:derivative},  $L_o\circ\sigma$ has a two--sided derivative at $s$, when either
there is only one minimizing geodesic joining $o$ to $\sigma(s)$, in particular when $\sigma(s) \notin C(o)$, or if there is more than one geodesic in $Geod(o,\sigma(s))$ which all make the same angle with the tangent vector $\sigma^\prime(s)$.
\end{remark}
 
\subsection{The Tensor Field $S$}\label{s:TFS}
 As in \cite{HI2} we define a symmetric tensor field $S$ of type $(2,0)$ on $(M\backslash( \{o\} \cup C(o)))  \cap \{p : L_o(p) <\ell\}$ by
$$  S = \frac{y^\prime\circ L_o}{y\circ L_o} \left(g - dL_o\otimes dL_o\right) - \nabla^2L_o$$
where $g(-,-) = \langle-,-\rangle$ denotes the Riemannian metric on $M$. Because by \cite[Proposition 2.20]{GW} the Hessian of $L_{\tilde o}$ satisfies
$$ \nabla^2L_{\tilde o} = \frac {y^\prime(r)}{y(r)}\left( ds^2-dr^2\right),$$
we see that in a certain sense,  $S$ measures the difference between the Hessians of $L_{\tilde o}$ and $L_o$. 
For vector fields $X$ and $Y$ on $(M\backslash( \{o\} \cup C(o)))  \cap \{p : L_o(p) <\ell\}$, we have
$$S(X,Y) =  \frac{y^\prime\circ L_o}{y\circ L_o}( \langle X,Y\rangle - \langle X,\xi\rangle\langle Y,\xi\rangle) - \langle \nabla_X\xi, Y\rangle , $$
where $\xi = grad(L_o)$ is the unit radial vector field on $M \backslash (\{o\} \cup C(o))$ and because $\nabla^2L_o(X,Y)= \langle \nabla_X\xi, Y\rangle$.

\begin{lemma} \label{l:S} %2.5
If $(\widetilde M, \tilde o)$ has stronger radial attraction than $(M,o)$ then
$$S(X,X) \leq 0$$
for all vector fields $X$ on  $(M\backslash (\{o\} \cup C(o)))  \cap \{p : L_o(p) <\ell\}$.
\end{lemma}

\begin{proof}
Let $ L_o(p) <\ell$  and $X \in T_pM$ be a unit tangent vector. Let $\sigma$ be the geodesic with $\sigma^\prime(0) =X$. Let $\tilde \sigma$ be a geodesic in
$\widetilde M$ satisfying $L_{\tilde o}( \tilde\sigma(0) ) = L_o(\sigma(0))$ and $(L_{\tilde o}\circ \tilde\sigma)_+^\prime(0) = (L_o\circ\sigma)_+^\prime(0)$.
Consider the smooth function  $f(t) = L_{\tilde o}\circ \tilde\sigma(t) - L_o\circ\sigma(t)$. Because $(\widetilde M, \tilde o)$ has stronger radial attraction than $(M,o)$, there exists an $\epsilon > 0$ such that $f(t)\leq 0$ for $0\leq t < \epsilon$.  Since $f(0)=f^\prime(0)=0$ it follows that $f^{\prime\prime}(0) \leq 0$.  But by \cite[(2.11)]{HI1},
$S(X,X) = f^{\prime\prime}(0)$.
\end{proof}

\begin{remark}%2.6
In \cite{HI1} we showed that the condition of weaker radial attraction is equivalent to the condition $S(X,X) \geq 0$ for all $X$. 
In Section \ref{s:curvature} we show that the converse of Lemma \ref{l:S} holds when $M$ is a real analytic Riemannian manifold.
It remains an open question whether the converse of Lemma \ref{l:S} holds in the $C^\infty$ case.
\end{remark}

\subsection{Conjugate cut points}

\begin{proposition}\label{p:ccp}%2.7
Suppose that $(\widetilde M, \tilde o)$ has stronger radial attraction than $(M,o)$.  If $q \in C(o)$ is a conjugate cut point of $o$, then $d(o,q) \geq \ell$. Thus $C(o) \cap \{q: d(o,q)<\ell\}$ consists entirely of nonconjugate cut points or is empty.
\end{proposition}
\begin{proof}
If not, suppose that $q$ is conjugate to $o$ along the unit speed geodesic $\gamma : [0,a] \rightarrow M $ joining $o$ to $q$ with $ a<\ell$.  There exists a nontrivial Jacobi field $J$ along $\gamma$ with $J(0) =0$, $J(a) = 0$, and $J^\prime(0)$ a unit vector perpendicular to $\gamma^\prime(0)$.  Let $ \tilde \gamma : [0,a] \rightarrow \widetilde M$ be a geodesic emanating from $\tilde o$. Let $\tilde J$ be the Jacobi field along $\tilde\gamma$ satisfying $\tilde J(0) =0$ and $\tilde J^\prime(0)$ a unit vector perpendicular to $\tilde\gamma^\prime(0)$.  Thus $ |\tilde J(t)| = y(t)$ for $0\leq t <\ell$. By assumption,  $(\widetilde M, \tilde o)$ has stronger radial attraction than $(M,o)$. Thus, by Lemma \ref{l:S}, the Hessians satisfy $\nabla^2 L_{\tilde o} \leq \nabla^2 L_o$ along $\gamma$ and $\tilde\gamma$ for all $0< t <a $. By \cite[Theorem 3.13]{HI2}, it follows that $y(t) = |\tilde J(t)| \leq |J(t)|$ for all $0\leq t \leq a$.  This is a contradiction since $0 < y(a) \leq |J(a)| =0$.
 \end{proof}
 
 \begin{remark}%2.8
Proposition \ref{p:ccp} implies that if $\ell = \infty$, then $C(o)$ contains only  nonconjugate cut points.
 \end{remark}
 
\section{Triangle comparison}\label{s:tct}

\subsection{Triangles} \label{s:tri}
We are concerned with geodesic triangles $\triangle opq$ in the pointed Riemannian manifold $(M, o)$,
 where $o$ is the base point of $M$ and $p \neq q$ are two other points of $M$ distinct from $o$.  The side $\overline {pq}$ of $\triangle opq$ joining $p$ to $q$ is a unit speed minimizing geodesic  $\sigma: [0,d(p,q)] \rightarrow M$.  Similarly the sides $\overline{op}$ and $\overline{oq}$ are  unit speed minimizing geodesics $\tau:[0,d(o,p)] \rightarrow M$ and $\gamma:[0,d(o,q)] \rightarrow M$ respectively.  (See Figure \ref{f:opq}.) There may be ambiguity in the notation $\triangle opq$ when $p$ or $q$ is a cut point of $o$, or when $q$ is a cut point of $p$, since there may be more than one minimizing geodesic joining the pair of vertices in any of these cases.  However, when we use the notation $\triangle opq$ we will implicitly assume that appropriate choices of the sides have been made so that the ambiguity in the notation should cause no problems. Indeed, we implicitly assume that $\tau$ has been chosen so that $\langle \tau^\prime(d(0,p)), \sigma^\prime(0)\rangle = (L_o\circ \sigma)_+^\prime(0)$ and that $\gamma$ has been chosen so that 
 $\langle \gamma^\prime(d(0,q)), \sigma^\prime(d(p,q))\rangle = (L_o\circ \sigma)_-^\prime(d(p,q))$ in accordance with Lemma \ref{l:derivative}.
 
 \begin{figure}[htbp]
\begin{center}
\setlength{\unitlength}{1cm}
\begin{picture}(5,5)
\put(1,0){\line(0,1){4}}\put(0.7,2){$\tau$}
\put(1,0){\line(3,1){3}}\put(2.5,0.2){$\gamma$}
\put(1,4){\line(1,-1){3}}\put(2.5,2.6){$\sigma$}
\put(0.7,-0.1){$o$}
\put(0.7,4){$p$}
\put(4.1,1){$q$}
\end{picture}
\caption{$\triangle opq$.}
\label{f:opq}
\end{center}
\end{figure}

 A triangle $\triangle \tilde o\tilde p\tilde q$ in $(\widetilde M,\tilde o)$ corresponds to $\triangle opq$ if the corresponding sides have the same lengths, that is, if 
 \begin{equation}\label{e:correspond}
 d(\tilde o,\tilde p) =d(o,p), \quad d(\tilde o,\tilde q) = d(o,q),\quad \mathrm{and}\quad d(\tilde p,\tilde q) = d(p,q).
 \end{equation}
 Because of the constraints on the distance function in $\widetilde M$, the corresponding triangle in $\widetilde M$ to a given triangle $\triangle opq$ exists  if and only if 
 \begin{equation}\label{e:exist}
d(o,p) <\ell, \quad d(o,q)<\ell,\quad \mathrm{and}\quad d(p,q)\leq D_\pi(L_o(p),L_o(q)).
\end{equation}
In particular, we always assume (\ref{e:correspond}) and (\ref{e:exist}) hold when we say $\triangle \tilde o\tilde p\tilde q$ corresponds to $\triangle opq$.  Because of (\ref{e:exist}), $ \tilde p, \tilde q \notin C(\tilde o)$. Hence there is no ambiguity in the choice of the minimizing geodesics $\tilde \tau$ joining $\tilde o$ to $\tilde p$ and  $\tilde \gamma$ joining $\tilde o$ to $\tilde q$. However, if $\tilde q \in C(\tilde p)$, there may be more than one choice of a minimizing geodesic joining $\tilde p$ to $\tilde q$, and we always choose $\tilde \sigma$ to be the uppermost one to avoid any ambiguity. (See \cite[Section 2]{HI2}.) In geodesic polar coordinates 
$(r,\theta)$ we always assume that $\tilde p$ lies on the $\theta =0$ meridian and that  $\tilde q$ lies on the meridian with $0\leq \theta\leq \pi$.  

\subsection{Fundamental Lemma} The definition of stronger radial attraction is formulated to give a short proof of the following lemma.

\begin{lemma} \label{l:basic} %3.1
Suppose $(\widetilde M, \tilde o)$ has stronger radial attraction than $(M,o)$ and that $\tilde\sigma:[a,b]\rightarrow \widetilde M$ and $\sigma
:[a,b]\rightarrow M$ are unit speed geodesics.   Assume $\ell > L_{\tilde o}(\tilde\sigma(t)) \geq L_o(\sigma(t))$ for all $a \leq t \leq b$. Moreover, assume that $L_o\circ\sigma$ has a 2--sided derivative for all $a < t < b$. If there exists a $t_0 \in (a,b)$ where $L_{\tilde o}(\tilde\sigma(t_0)) = L_o(\sigma(t_0))$, then 
$L_{\tilde o}(\tilde\sigma(t)) = L_o(\sigma(t))$ for all $t \in [a,b]$.
\end{lemma}

\begin{proof}
One shows that the set $\{ t \in(a,b) : L_{\tilde o}(\tilde\sigma(t)) = L_o(\sigma(t))\}$ is both open and closed.  It is clearly closed by continuity. On the other hand, if  $L_{\tilde o}(\tilde\sigma(t_0)) = L_o(\sigma(t_0))$, then we also have that $(L_{\tilde o}\circ\tilde\sigma)^\prime(t_0) =( L_o\circ\sigma)^\prime(t_0)$ since the difference function $L_{\tilde o}\circ \tilde \sigma - L_o\circ \sigma$ attains a minimum at $t_0$.
Apply the stronger radial attraction hypothesis to the left and right of $t_0$ to see that $L_{\tilde o}(\tilde\sigma(t)) = L_o(\sigma(t))$ for $t$ in a neighborhood of $t_0$.
\end{proof}

\subsection{Thin Triangles}

\begin{definition}%3.2
A geodesic triangle $\triangle opq$ in $M$ is \emph{thin} if $L_o(p) + d(p,q) < \ell$, $d(p,q)$ is less than the injectivity radius $inj_{\widetilde M}(\tilde p)$ of $\tilde p$ in $\widetilde M$, and $L_o\circ\sigma$ has a two--sided derivative at all interior points of the side $\sigma$ joining $p$ to $q$. The point $\tilde p$ in $\widetilde M$ is chosen so that $L_{\tilde o}(\tilde p ) = L_o(p)$. 
\end{definition}

\begin{proposition}\label{p:thin}%3.3
Suppose that $(\widetilde M, \tilde o)$ has stronger radial attraction than $(M,o)$.  If $\triangle opq$ is thin, then the corresponding geodesic triangle
$\triangle \tilde o\tilde p\tilde q$ exists.  Moreover
if $\sigma : [0, d(p,q)] \rightarrow M$ and $\tilde\sigma : [0, d(p,q)] \rightarrow \widetilde M$ are the sides of $\triangle opq$ and $\triangle \tilde o\tilde p\tilde q$ joining $p$ to $q$ and $\tilde p$ to $\tilde q$ respectively, then
$$ L_o(\sigma(t)) \leq L_{\tilde o}(\tilde\sigma(t))$$
for  $0\leq t \leq d(p,q)$.
\end{proposition}

\begin{proof} 
To prove existence of the corresponding triangle one checks that (\ref{e:exist}) holds:
(1) $d(o,p) = L_o(p) < L_o(p) + d(p,q) < \ell$, (2) $d(o,q) \leq L_o(p) + d(p,q) < \ell$ by the triangle inequality, and (3) $d(p,q) <  D_\pi(L_o(p),L_o(q))$ since $$d(p,q) < \min\{\ell-L_o(p), inj_{\widetilde M}(\tilde p)\}$$ implies that $\tilde q$ lies on a meridian with $\theta < \pi$ by Proposition \ref{p:monotonicity}.

Let $r_1 = L_o(p)$ and pick $\tilde p$ so that $L_{\tilde o}(\tilde p) = r_1$. For each $\phi \in [0, \pi]$, let $\tilde \sigma _\phi : [0, inj_{\widetilde M}(\tilde p)] \rightarrow \widetilde M$ be the unit speed geodesic emanating from $\tilde p$ making an angle $\phi$ with the meridian through $\tilde p$. By Proposition \ref{p:monotonicity}, the function $R_\phi(r_1, t) = L_{\tilde o}(\tilde \sigma_\phi(t))$ is strictly decreasing as a function of $\phi \in[0,\pi]$ for each fixed $t \in (0,inj_{\widetilde M}(\tilde p))$.
Clearly there exists a $\phi_0$ such that $(L_o\circ \sigma)_+^\prime(0) =(L_{\tilde o}\circ \tilde\sigma_{\phi_0})^\prime(0)$.  Thus by stronger radial attraction
$$ L_o(\sigma(t))\geq L_{\tilde o}( \tilde\sigma_{\phi_0}(t)) = R_{\phi_0}(r_1,t)$$ 
for all $0 \leq t <\epsilon$ for some $\epsilon > 0$.

Because $L_o(p) + d(p,q) < \ell$ and by the triangle inequality
we have (with $\phi =0$)
$$R_0(r_1,t) = L_{\tilde o}(\tilde p) + t= L_o(p) +t \geq  L_o(p)+ d(p,\sigma(t)) \geq L_o(\sigma(t))$$
for all $ 0 \leq t \leq d(p,q)$. Set $\bar \phi = \sup \{ \phi \in [0, \phi_0] : R_\phi(r_1, t)\geq  L_o(\sigma(t)) \forall t\in [0,d(p,q)]\}$.
By continuity, $L_{\tilde o}(\tilde\sigma_{\bar \phi}(t)) \geq L_o(\sigma(t))$ for all $ t\in [0,d(p,q)]$.  

We will argue that, setting $\tilde q = \tilde\sigma_{\bar\phi}(d(p,q))$,
we have $L_{\tilde o}(\tilde q) = L_o(q)$ and  $\triangle\tilde o\tilde p\tilde q$ is the desired comparison triangle.
There are four cases to consider. (1) If $L_{\tilde o}(\tilde\sigma_{\bar \phi} (d(p,q))) = L_o(q)$ we are done. (2) If there exists a $t^\ast \in (0, d(p,q))$ such that
$L_{\tilde o} ( \tilde\sigma_{\bar \phi}(t^\ast)) = L(\sigma(t^\ast))$ then  by  Lemma \ref{l:basic} we have $L_{\tilde o} ( \tilde\sigma_{\bar \phi}(t)) = L(\sigma(t))$
for all $0 \leq t \leq d(p,q)$. In particular $L_{\tilde o}(\tilde\sigma_{\bar \phi} (d(p,q))) = L_o(q)$ and we are done. (3) If $ \bar\phi = \phi_0$, then by stronger radial  attraction we have for some $\epsilon >0$, $L_{\tilde o}(\tilde\sigma_{\bar\phi}(t)) \leq L_o(\sigma(t)) $ for $0 \leq t < \epsilon$. Hence, combining with the opposite inequality for $t \in [0,d(p,q)]$ gives $L_{\tilde o}(\tilde\sigma_{\bar\phi}(t)) = L_o(\sigma(t)) $ for $0 \leq t <\epsilon$. Thus application of  Lemma \ref{l:basic} shows $L_{\tilde o} ( \tilde\sigma_{\bar \phi}(t)) = L(\sigma(t))$
for all $0 \leq t \leq d(p,q)$. In particular $L_{\tilde o}(\tilde\sigma_{\bar \phi} (d(p,q)) = L_o(q)$ and we are done. (4) Finally assume $ \bar\phi < \phi_0$ and 
$L_{\tilde o}(\tilde\sigma_{\bar \phi}(t)) > L_o(\sigma(t))$ for all $ t\in (0,d(p,q)]$.  Pick $\phi^\ast$ with $\bar\phi < \phi^\ast < \phi_0$. Since $( L_{\tilde o} \circ \tilde\sigma_{\phi^\ast} )^\prime(0)> (L_o\circ\sigma)^\prime_+(0)$ there exists an $\eta>0$ such that $L_{\tilde o}(\tilde\sigma_{\phi^\ast}(t)) >  L_o(\sigma(t))$
for all $ 0 < t <\eta$. Hence by monotonicity of $ R_\phi(r_1, t)$ as a function of $\phi$, it follows that 
$$L_{\tilde o}(\tilde\sigma_\phi(t)) = R_\phi(r_1, t) \geq R_{\phi^\ast}(r_1,t) =L_{\tilde o}(\tilde\sigma_{\phi^\ast}(t)) > L_o(\sigma(t))$$ 
for all  $\phi \leq \phi^\ast$ and $ 0 < t <\eta$. Our  assumption (4) implies that $L_{\tilde o}(\tilde\sigma_{\bar \phi}(t)) > L_o(\sigma(t))$ for all $ t\in [\eta,d(p,q)]$. In particular by continuity and compactness, we can slightly increase $\phi$ from $\bar\phi$ to find a $\hat\phi$ satisfying $\bar\phi < \hat\phi < \phi^\ast$ such that 
 $L_{\tilde o}(\tilde\sigma_{\hat \phi}(t)) > L_o(\sigma(t))$ for all $ t\in [\eta,d(p,q)]$.  Hence $L_{\tilde o}(\tilde\sigma_{\hat \phi}(t)) \geq L_o(\sigma(t))$ for $t \in[0,d(p,q)]$ which contradicts the choice of $\bar\phi$. Thus the last case (4)  is impossible and the proof is complete.
\end{proof}

Proposition \ref{p:thin} allows us to estimate the size of $\epsilon$ in the definition of stronger radial attraction.

\begin{corollary}\label{c:UE}%3.4
Suppose that $(\widetilde M, \tilde o)$ has stronger radial attraction than $(M,o)$, and let $\sigma$ and $\tilde \sigma$ be unit speed geodesics in $M$ and $\widetilde M$ respectively such that $L_o\circ \sigma(0) = L_{\tilde o}\circ \tilde \sigma(0)<\ell$ and $(L_o\circ \sigma)^\prime_+(0) = (L_{\tilde o}\circ \tilde \sigma)^\prime_+(0)$. Suppose  $L_o\circ \sigma(t)$ has a 2--sided derivative for all $0<t<b$.  If 
$0< \epsilon < \min\{ b, \ell-L_o(\sigma(0)),inj_{\widetilde M}(\tilde \sigma(0)) \}$,
then
$L_o\circ \sigma(t) \geq L_{\tilde o}\circ \tilde\sigma(t)$ for all $0 \leq t < \epsilon$.
\end{corollary}

\begin{proof} To simplify notation set $p = \sigma(0)$ and $\tilde p = \tilde\sigma(0)$. As in the proof of Proposition \ref{p:thin}, let $\phi_0$ be the angle that 
the tangent vector $\tilde\sigma^\prime(0)$ makes with the meridian.
By choice of $\epsilon$, for every $0<t_0 <\epsilon$, the triangle $\triangle op\sigma(t_0)$ is thin. Hence by Proposition \ref{p:thin},
there exists a corresponding triangle $\triangle \tilde o\tilde p\tilde q_{t_0}$ such that 
$$ L_o(\sigma(t)) \leq L_{\tilde o}(\tilde\sigma_{t_0}(t))$$
for  $0\leq t \leq t_0$ where  $\tilde\sigma_{t_0}$ is the side joining $\tilde p$ to $\tilde q_{t_0}$.
It follows that 
$$(L_{\tilde o}\circ\tilde\sigma_{t_0})_+^\prime(0) \geq (L_o\circ \sigma)^\prime_+(0)=(L_{\tilde o}\circ \tilde \sigma)^\prime_+(0)$$ 
which implies that the angle  made by the tangent vector $\tilde\sigma_{t_0}^\prime(0)$  with the meridian is less than or equal to $\phi_0$.  Hence by monotonicity of the function $R_\phi$, 
$$ L_o\circ \sigma(t_0) = L_{\tilde o}\circ \tilde\sigma_{t_0}(t_0) \geq L_{\tilde o}\circ\tilde\sigma(t_0).$$
Since $t_0 \in (0,\epsilon)$ is arbitrary, that concludes the proof. 
\end{proof}

\begin{remark}\label{r:thin}
The proof of Corollary \ref{c:UE} shows that if for every thin  $\triangle opq$ and its corresponding $\triangle\tilde  o\tilde p\tilde q$,
$$ L_o(\sigma(t)) \leq L_{\tilde o}(\tilde\sigma(t))$$
for  $0\leq t \leq d(p,q)$ , then $(\widetilde M, \tilde  o)$ has stronger radial attraction than $(M,o)$.
\end{remark}

\subsection{Bad Encounters} Analogous to the term \emph{bad encounter for the weaker radial attraction} defined in \cite[Section 4]{HI2}, we define \emph{bad encounter for the stronger radial attraction}, but before doing so we need to review some notation. 

Let $p\neq o$ be a point in the pointed space $(M,o)$ and $\tilde p$ in $\widetilde M$ with $d(\tilde o, \tilde p) =d(o,p)$.  Define the reference map $F : M \rightarrow \mathbf{R}^2$ by
$F(q)=(d(p,q),d(o,q))$ for $q\in M$ and the reference map $\widetilde F :\widetilde  M \rightarrow \mathbf{R}^2$ by $\widetilde F(\tilde q) = (d(\tilde p,\tilde q),d(\tilde o,\tilde q))$ for $\tilde q \in \widetilde M$.  More about the reference maps can be found in \cite{ISU} and \cite{HI2}.

The structure of the cut locus  of $\tilde p$ in $\widetilde M$ is described in \cite[Section 2]{HI2}. To summarize: $C(\tilde p)$ is a tree.  The trunk of the tree is the portion of $C(\tilde p)$ contained in the opposite meridian from $\tilde p$, 
and the positive branches are the connected components of $C(\tilde p) \cap int(\widetilde M^+)$. If $\tilde q$ lies in a positive branch,  there is an arc $\alpha$ in the branch joining $\tilde q$ to the trunk which is parameterized by distance from $\tilde p$. Thus $\alpha$ defines a homeomorphism into $C(\tilde p)$ defined for $t$ in some interval $d(\tilde p, \tilde q) \leq t < b\leq \infty$ such that $d(\tilde p, \alpha(t))=t$.

\begin{definition}\label{d:badencounter}
Let $\sigma : [0,l] \to M$ be a minimizing geodesic in $M$ emanating from $p \in M$.  
We say that $\sigma$ has an \emph{encounter with the cut locus} of $\tilde p$ at $t_0 \in (0,l)$ if
$F(\sigma(t_0)) \in \widetilde F(C(\tilde p)\cap int(\widetilde M^+))$. 
\end{definition}

\begin{definition}
 Suppose that $\tilde q $ is the unique point in $C(\tilde p) \cap int(\widetilde M^+)$  such that $\widetilde F(\tilde q) = F(\sigma(t_0))$ and $\alpha$ is the arc in $C(\tilde p)$ joining $\tilde q$ to the trunk. The encounter at $t_0$ is a \emph{bad encounter for stronger radial attraction} if for every $\epsilon >0$ there exists $ t^\ast \in (t_0, t_0+\epsilon)$ such that $L_o ( \sigma(t^\ast)) < L_{\tilde o}(\alpha(t^\ast))$.     If no confusion arises from the context, we may simply say \emph{bad encounter}. 
\end{definition}

\subsection{Triangle Comparison Theorem}

 \begin{theorem} \label{t:TCT}%3.8
 Suppose that the model surface $(\widetilde M,\tilde o)$ has stronger radial attraction than $(M,o)$.
 Let $\triangle opq$ be a geodesic triangle in $M$ for which there exists a corresponding triangle $\triangle \tilde o\tilde p\tilde q$ in $\widetilde M$, that is, Equations (\ref{e:correspond}) and (\ref{e:exist}) hold. Let $\sigma: [0, d(p,q)] \rightarrow M$ be the side joining $p$ to $q$.  Assume that (1) $L_o\circ\sigma(t)$ has a 2--sided derivative for all $0<t<d(p,q)$ and (2) $\sigma$ has   no bad encounters with $C(\tilde p)$.  If $\tilde \sigma$ is the side joining $\tilde p$ to $\tilde q$, then
\begin{equation}\label{e:AC}
 L_o\circ \sigma (t) \leq  L_{\tilde o} \circ \tilde \sigma(t)
 \end{equation}
 for all $0\leq t \leq d(p,q)$.
 \end{theorem}
 
 \begin{proof}
 Since Equations (\ref{e:correspond}) and (\ref{e:exist}) hold, we have $F(q) \in \widetilde F(\widetilde M)$.
 If $r_0 = d(\tilde o,\tilde p)$, then by \cite[Section 2]{HI2},
 $$\widetilde F(\widetilde M) \subset R \equiv \{ (x,y) : r_0 \leq x+y \leq  2\ell-r_0, -r_0 \leq y-x \leq r_0\}.$$
If $F(q)$ lies in the boundary of  $R$ then  $\triangle\tilde o \tilde p\tilde q$ is degenerate and
(\ref{e:AC}) holds without any extra hypothesis.
 
\begin{lemma}\label{l:degenerate}%3.9
Suppose $F(q) \in \partial R$, then (\ref{e:AC}) holds.
\end{lemma} 
\begin{proof}
Set $F(q) = (x_0,y_0)$.  We consider each case separately. See Figure \ref{f:degenerate}.

(i) If $x_0+y_0=r_0$, then $(L_o\circ\sigma)^\prime (0) = -1$, and $q$ lies on the minimal geodesic joining $o$ to $p$.  Clearly, $L_o\circ\sigma(t) = r_0 - t = L_{\tilde o}\circ \tilde \sigma(t)$ for $0\leq t \leq d(p,q)$.

(ii) If $y_0-x_0 = r_0$, then $(L_o\circ\sigma)^\prime (0) = 1$, and $p$ lies on the minimizing geodesic joining $o$ to $q$.
Clearly, $L_o\circ\sigma(t) = r_0 + t = L_{\tilde o}\circ \tilde \sigma(t)$ for $0\leq t \leq d(p,q)$.

(iii) If $y_0-x_0 = -r_0$, then $(L_o\circ\sigma)^\prime (0) = -1$ and $o$ lies on the minimal geodesic joining $p$ to $q$.
Clearly 
$L_o\circ\sigma(t) =| t-r_0| = L_{\tilde o}\circ\tilde \sigma(t)$ for $0\leq t \leq d(p,q)$, even though $L_o\circ\sigma$ does not have a 2--sided derivative at $t=r_0$.

(iv) If  $x_0+y_0 = 2\ell-r_0$, then  two applications of the triangle inequality:
$$d(o,\sigma(t))\leq d(o,p)+d(p,\sigma(t))=r_0+t\quad \mathrm{for}\quad 0\leq t\leq \ell-r_0$$ 
and
$$d(o,\sigma(t))\leq d(o,q)+d(q,\sigma(t))=y_0 + x_0-t = 2\ell-r_0 -t  $$
for $\ell-r_0\leq t\leq d(p,q)$,
imply that
$$L_o\circ \sigma(t) \leq L_{\tilde o}\circ\tilde \sigma(t) = 
\left\{ 
\begin{array}{ccl}r_0+t&\mathrm{if}&0\leq t \leq \ell-r_0\\ 2\ell-r_0-t &\mathrm{if}&\ell-r_0 \leq t\leq d(p,q).\end{array}
\right.
$$
 \end{proof}
 
 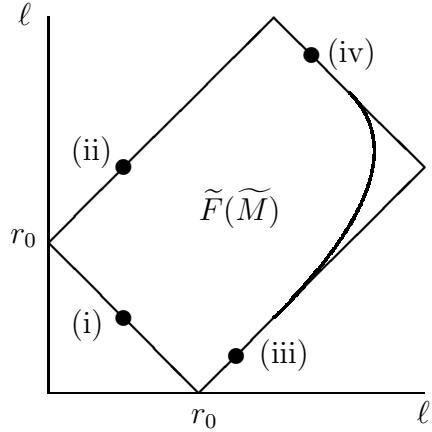
\begin{figure}[htbp]
\begin{center}
\setlength{\unitlength}{1cm}
\begin{picture}(6,6)(-1,-1)
%{\color{gray}
\put(0,0){\line(1,0){5}} %x axis
\put(0,0){\line(0,1){5}} %y axis
%}
\put(-0.5,2){$r_0$}
\put(-0.4,4.9){$\ell$}
\put(1.9,-0.4){$r_0$}
\put(4.9,-0.4){$\ell$}
\thicklines %Ftilde(Mtilde)
\put(0,2){\line(1,-1){2}}
\put(0,2){\line(1,1){3}}
\put(2,0){\line(1,1){3}}
\put(3,5){\line(1,-1){2}}
\qbezier(3,1)(5,3)(4,4)
\put(2.0,2.3){$\widetilde F(\widetilde M)$}
\put(1,1){\circle*{0.2}} \put(0.3,0.8){(i)}
\put(1,3){\circle*{0.2}}\put(0.3,3.1){(ii)}
\put(3.5,4.5){\circle*{0.2}}\put(3.7,4.4){(iv)}
\put(2.5,0.5){\circle*{0.2}}\put(2.8,0.4){(iii)}
\end{picture}
\caption{The set $R$ containing $\widetilde F(\widetilde M)$ with possible positions for $F(q)$ on $\partial R$ which give  degenerate triangles corresponding to cases (i) through (iv). In this figure  $ F(p)= (0,r_0) $ and $F(o)=(r_0,0)$.}
\label{f:degenerate}
\end{center}
\end{figure}

 From now on assume that $F(q)=(x_0,y_0)$ lies in the interior of $R$. As in the proof of Proposition \ref{p:thin}, for each $\phi \in [0,\pi]$, let $\tilde \sigma_\phi$ be the unit speed maximal minimizing geodesic emanating from $\tilde p$ making the angle $\phi$ with the meridian through $\tilde p$. Let $\phi_0$ be the angle such that $\tilde \sigma_{\phi_0} = \tilde\sigma$.  
 
 If $t_\phi^c$ denotes distance to the cut point along $\tilde\sigma_\phi$, then $\tilde\sigma_\phi$ is defined on the interval 
 $[0,t_\phi^c]$.
 Let $\alpha_\phi$ be the arc in $C(\tilde p)$ joining $\tilde \sigma_\phi(t_\phi^c)$ to the trunk. ($\alpha_\phi$ may be empty if 
 $\tilde \sigma_\phi(t_\phi^c)$ is in the trunk.)
 We define
 $\varsigma_\phi$ to be the concatenated curve $\tilde\sigma_\phi\cdot \alpha_\phi$.
 (cf. \cite{HI2}.)
 If $\phi  \leq \phi_0$, then the curve $\widetilde F(\varsigma_\phi)$ lies above the curve $\widetilde F(\tilde\sigma)$  on the interval $[0, d(p,q)]$, so that there exists a parameter value $\check{t}_\phi \in(0,d(p,q))$ 
 where $\widetilde F(\varsigma_\phi)$ crosses the line
 $x+y= x_0+y_0$ in the reference space $\widetilde F( \widetilde M)$.
 This leads to the definition for each $ 0 \leq \phi \leq \phi_0$,
\begin{equation*}
 f_\phi (t) =
\left\{
\begin{array}{ll}
  d(\tilde o,\varsigma_\phi(t)) & \mathrm{if\enspace} 0\leq t \leq \check t_\phi    \\
x_0+y_0 -t &  \mathrm{if\enspace} \check t_\phi \leq t \leq d(p,q).    
\end{array}
\right.
\end{equation*}
See Figure \ref{f:fphi}.
Thus $f_\phi(t)$ is continuous in $\phi$ and $t$.
Moreover  $f_{\phi_0}(t) = L_{\tilde o}\circ \tilde\sigma(t)$.  Therefore our goal is to prove   $f_{\phi_0}(t) \geq f(t)= L_o\circ\sigma(t)$ for $ 0 \leq t \leq d(p,q)$. 

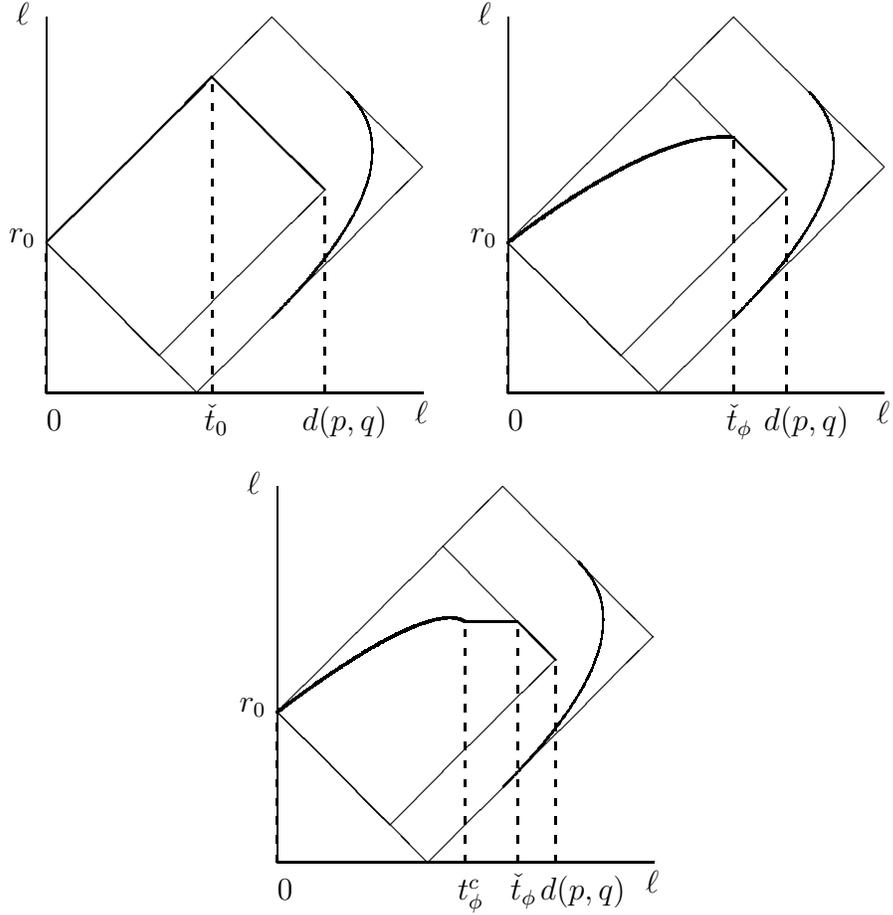
\begin{figure}[htbp]
\setlength{\unitlength}{1cm}
\begin{picture}(6,6)(-1,-1)%phi=0
\thinlines
%{\color{gray}
\put(0,0){\line(1,0){5}} %x axis
\put(0,0){\line(0,1){5}} %y axis
%R
\put(0,2){\line(1,-1){2}}
\put(0,2){\line(1,1){3}}
\put(2,0){\line(1,1){3}}
\put(3,5){\line(1,-1){2}}
\qbezier(3,1)(5,3)(4,4)
%rectangle
\put(1.5,0.5){\line(1,1){2.2}}
\put(0,2){\line(1,-1){1.5}}
%} ends color
\put(-0.5,2){$r_0$}
\put(-0.4,4.9){$\ell$}
%\put(1.9,-0.4){$r_0$}
\put(4.9,-0.4){$\ell$}
%f_0
\thicklines
\multiput(0,0)(0,.25){8}{\line(0,1){.1}}\put(0,-0.5){$0$}
\multiput(3.7,0)(0,.25){11}{\line(0,1){.1}}\put(3.4,-0.5){$d(p,q)$}
\multiput(2.2,0)(0,.25){17}{\line(0,1){.1}}\put(2.1,-0.5){$\check t_0$}
\thicklines
\put(0,2){\line(1,1){2.2}}
\put(2.2,4.2){\line(1,-1){1.5}}
\end{picture}
\begin{picture}(6,6)(-1,-1) % t-check < t cut
\thinlines
%{\color{gray}
\put(0,0){\line(1,0){5}} %x axis
\put(0,0){\line(0,1){5}} %y axis
%R
\put(0,2){\line(1,-1){2}}
\put(0,2){\line(1,1){3}}
\put(2,0){\line(1,1){3}}
\put(3,5){\line(1,-1){2}}
\qbezier(3,1)(5,3)(4,4)
%rectangle
\put(2.2,4.2){\line(1,-1){1.5}}
\put(1.5,0.5){\line(1,1){2.2}}
\put(0,2){\line(1,-1){1.5}}
%} ends color
\put(-0.5,2){$r_0$}
\put(-0.4,4.9){$\ell$}
%\put(1.9,-0.4){$r_0$}
\put(4.9,-0.4){$\ell$}
%f_phi
\thicklines
\multiput(0,0)(0,.25){8}{\line(0,1){.1}}\put(0,-0.5){$0$}
\multiput(3.7,0)(0,.25){11}{\line(0,1){.1}}\put(3.4,-0.5){$d(p,q)$}
\multiput(3,0)(0,.25){14}{\line(0,1){.1}}\put(2.9,-0.5){$\check t_\phi$}
\qbezier(0,2)(2,3.5)(3,3.4)
\put(3,3.4){\line(1,-1){0.7}}
\end{picture}
\begin{center}
\begin{picture}(6,6)(-1,-1) % t-check > t cut
\thinlines
%{\color{gray}
\put(0,0){\line(1,0){5}} %x axis
\put(0,0){\line(0,1){5}} %y axis
%R
\put(0,2){\line(1,-1){2}}
\put(0,2){\line(1,1){3}}
\put(2,0){\line(1,1){3}}
\put(3,5){\line(1,-1){2}}
\qbezier(3,1)(5,3)(4,4)
%rectangle
\put(2.2,4.2){\line(1,-1){1.5}}
\put(1.5,0.5){\line(1,1){2.2}}
\put(0,2){\line(1,-1){1.5}}
%} ends color
\put(-0.5,2){$r_0$}
\put(-0.4,4.9){$\ell$}
%\put(1.9,-0.4){$r_0$}
\put(4.9,-0.4){$\ell$}
\thicklines 
\multiput(0,0)(0,.25){8}{\line(0,1){.1}}\put(0,-0.5){$0$}
\multiput(3.7,0)(0,.25){11}{\line(0,1){.1}}\put(3.5,-0.5){$d(p,q)$}
\multiput(3.2,0)(0,.25){13}{\line(0,1){.1}}\put(3.1,-0.5){$\check t_\phi$}
\multiput(2.5,0)(0,.25){13}{\line(0,1){.1}}\put(2.4,-0.5){$ t^c_\phi$}
\qbezier(0,2)(2,3.5)(2.5,3.2)
\put(2.5,3.2){\line(1,0){0.7}}
\put(3.2,3.2){\line(1,-1){0.5}}
\end{picture}
\end{center}
\caption{Showing typical graphs of $f_\phi$ for $\phi =0$ (upper left),  $t_\phi^c > \check t_\phi$ (upper right), and  $t_\phi^c < \check t_\phi$(bottom center) as curves in the reference space.}
 \label{f:fphi}
\end{figure}

Since $\triangle opq$ is nondegenerate, we have that $f_0 (t) > f(t)$ for $0< t < d(p,q)$, 
$f_0^\prime(0) = 1 > f^\prime(0)$,
and $f_\phi(t) > f(t)$ for $\check t_\phi \leq t < d(p,q)$ and $0 \leq \phi <\phi_0$. 
Set 
$$ \bar \phi = \sup \{ 0 \leq \phi  \leq \phi_0 : f^\prime(0) < f_\phi^\prime(0)\enspace\mathrm{and}\enspace f(t) < f_\phi(t) \enspace\mathrm{for}\enspace 0< t < d(p,q)\}.$$ 

If $ \bar \phi = \phi_0$, then $f \leq f_{\phi_0}$ and we are done. Hence we assume that $\bar \phi < \phi_0$. Then by continuity and compactness $f(t) \leq f_{\bar \phi}(t)$ for all $t$ and either $f^\prime(0) = f^\prime_{\bar\phi}(0)$ or there exists a $\bar t$, $0< \bar t< d(p,q)$, such that $f_{\bar\phi}(\bar t) = f(\bar t)$.

\begin{remark} \label{r:tbar}
We note that if such a $\bar t$ exists, it would be that $0< \bar t< \check t_{\bar\phi}$ for the following reason:
$F(\sigma(t))$  is confined in the rectangle $\{(x,y): r_0 < x+y < x_0+y_0,    y_0-x_0   < y-x < r_0\}$. (See \cite[Lemma 1.2]{HI2}.) Thus $f(t) < x_0+y_0 -t = f_{\bar \phi}(t)$ for $\check t_{\bar\phi} \leq t < d(p,q)$. See Figure \ref{f:confined},
\end{remark}

\begin{figure}[htbp]
\begin{center}
\setlength{\unitlength}{1cm}
\begin{picture}(6,6)(-1,-1)
%{\color{gray}
\put(0,0){\line(1,0){5}} %x axis
\put(0,0){\line(0,1){5}} %y axis
\put(0,2){\line(1,-1){2}}
\put(0,2){\line(1,1){3}}
\put(2,0){\line(1,1){3}}
\put(3,5){\line(1,-1){2}}
\qbezier(3,1)(5,3)(4,4)
%}
\put(-0.5,2){$r_0$}
\put(-0.4,4.9){$\ell$}
\put(1.9,-0.4){$r_0$}
\put(4.9,-0.4){$\ell$}
\qbezier(0,2)(1,1.5)(3.5,2.5)
\put(1.3,2.1){$F\circ\sigma$}
\thicklines %rectangle
\put(1.5,0.5){\line(1,1){2}}
\put(0,2){\line(1,-1){1.5}}
\put(0,2){\line(1,1){2}}
\put(2,4){\line(1,-1){1.5}}
\put(3.5,2.5){\circle*{0.2}} \put(3.5,2.5){$^{F(q)}$}
\end{picture}
\caption{Showing $F(\sigma(t))$ confined to the rectangle  $\{(x,y): r_0 < x+y < x_0+y_0,    y_0-x_0   < y-x < r_0\}$ where $F(q) = (x_0,y_0)$.}
\label{f:confined}
\end{center}
\end{figure}
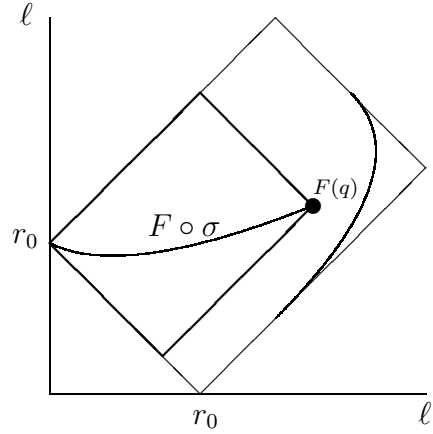

First we assume the existence of $\bar t$.
Thus either
$\varsigma_{\bar\phi}(\bar t)$ is a cut point of $C(\tilde p)$ on $\alpha_{\bar\phi}$, or it lies in the domain of $\tilde\sigma_\phi$.
 In the first case, $\sigma$ has an encounter with $C(\tilde p)$ at the parameter $\bar t$. Because there are no bad encounters, $f(t) \geq f_{\bar\phi}(t)$ for all $ \bar t \leq t \leq \check t_{\bar\phi}$ which contradicts $f(\check t_{\bar\phi}) < f_{\bar\phi}(\check t_{\bar\phi})$.  
Thus $\bar t$ must be a parameter value along $\tilde\sigma_{\bar\phi}$.  But then applying  Lemma \ref{l:basic}  with $\sigma$ and $\tilde\sigma_{\bar\phi}$, it follows
that $ f(t) = f_{\bar\phi}(t)$ for all $0 \leq t \leq \min( t^c_{\bar\phi}, \check t_{\bar\phi})$.
Depending upon which of  $t^c_{\bar\phi}$ or $ \check t_{\bar\phi}$ is the smaller, this leads either to an encounter with the cut locus at $t= t^c_{\bar\phi}$ which leads to a contradiction as above or to
$f(\check t_{\bar\phi}) = f_{\bar\phi}(\check t_{\bar\phi})$ which is impossible.  
Hence there is no such $\bar t$, and we would have the case $f_{\bar\phi}^\prime(0) = f^\prime(0)$. By stronger radial attraction we have for some $\epsilon>0$, $f_{\bar\phi}(t) \leq f(t)$ for all $0\leq t <\epsilon$.   But also $f(t) \leq f_{\bar \phi}(t)$ for all $t$. Hence  $f_{\bar\phi}(t) = f(t)$ for all $0\leq t <\epsilon$.  
Thus by Lemma \ref{l:basic}, this again leads to $ f(t) = f_{\bar\phi}(t)$ for all $0 \leq t \leq \min( t^c_{\bar\phi}, \check t_{\bar\phi})$ which we just saw was impossible. 

Therefore we conclude that $\bar \phi = \phi_0$ which completes the proof.

 \end{proof}

\begin{remark}
The condition of no bad encounters is automatically satisfied whenever the model $\widetilde M$ has the property that the cut locus of every point $\tilde p$ is a subset of the opposite meridian.  The paper \cite{TAS} presents sufficient conditions for a model to have the stated property.
\end{remark}

The following angle comparison holds as well.

\begin{corollary}\label{c:angle}%3.12
Under the assumptions of Theorem \ref{t:TCT}, we have $\measuredangle opq \leq \measuredangle \tilde o\tilde p\tilde q$ and $\measuredangle oq p\leq \measuredangle \tilde o \tilde q\tilde p$. 
\end{corollary}

\begin{proof} Since $L_o\circ\sigma(0)=L_{\tilde o}\circ\tilde\sigma(0)$ and $L_o\circ\sigma(d(p,q))=L_{\tilde o}\circ\tilde\sigma(d(p,q))$,
 Equation (\ref{e:AC}) implies $(L_o\circ\sigma)_+^\prime(0) \leq (L_{\tilde o}\circ\tilde\sigma)_+^\prime(0)$ and $(L_o\circ\sigma)_-^\prime(d(p,q)) \geq (L_{\tilde o}\circ\tilde\sigma)_-^\prime(d(p,q))$.  Therefore,  $\measuredangle opq \leq \measuredangle \tilde o\tilde p\tilde q$ and $\measuredangle oqp\leq \measuredangle \tilde o \tilde q\tilde p$ because of how  the sides of $\triangle opq$ are chosen as described in Section 
 \ref{s:tri} .
\end{proof}
 
\subsection{Examples}\label{x:cut}%3.6

The assumption that the side $\sigma$ joining $p$ to $q$ in $\triangle opq$ has a 2--sided derivative at all interior points cannot be removed from the hypothesis of   Theorem \ref{t:TCT}.

Let $\widetilde M$ be the Euclidean plane $\mathbf{R}^2$ with origin $\tilde o = (0,0)$, and let $M = S^1 \times \mathbf{R}$ be the flat cylinder with origin $o = ((1,0),0)$.  
(We regard $S^1$ as unit circle in $\mathbf{R}^2)$. Clearly $(\widetilde M, \tilde o)$ has stronger radial attraction than $(M,o)$. Let $p= ((\frac {-1}{\sqrt 2}, \frac 1 {\sqrt 2}),0)$ and $q= ((\frac {-1}{\sqrt 2}, \frac {-1} {\sqrt 2}),0)$. The side $\sigma$ of $\triangle opq$  joining $p$ to $q$ passes through $((-1,0),0)\in C(o)$. 
Then 
\begin{equation}
L_o(\sigma(t)) = \left\{ 
\begin{array}{ccc}
 \frac {3\pi}4 + t &\mathrm{if}& 0 \leq t \leq \frac \pi 4  \\
  \frac{5\pi}4-t&\mathrm{if}  &    \frac \pi 4 \leq t \leq \frac \pi 2
\end{array}
\right.
\end{equation}
which does not have a 2--sided derivative at $t= \frac {\pi}4$. The corresponding triangle $\triangle \tilde o\tilde p\tilde q$ has
 $\tilde p = (\frac{3\pi}4,0)$ and  $\tilde q =(\frac {7\pi}{12}, \frac {\sqrt 2 \pi}{3})$. Thus $\tilde\sigma(t) = \tilde p + t (\tilde q - \tilde p)/\frac \pi 2$.
Therefore $L_{\tilde o}(\tilde\sigma(t) = \sqrt{ ( \frac{3\pi}4 -\frac t 3)^2 + \frac{8t^2}9} $ for $ 0 \leq t \leq \frac \pi 2$. One easily checks that $L_{\tilde o}(\tilde\sigma(t)) < L_o(\sigma(t))$ for $0<t<\frac \pi 2$, which is the wrong inequality.  

This example can be generalized whenever there exists a cut point of $o$ with distance less than $\ell$. For let $r < \ell$ be the distance from $o$ to its nearest cut point.  Using \cite[Lemma 5.6]{CE}, there exists a minimal geodesic loop $\gamma$ based at $o$ such that $\gamma(r)$ is the nearest cut point to $o$.  Clearly, we can pick a $ \delta >0$ sufficiently small such that if $p=\gamma(r-\delta)$ and $q=\gamma(r+\delta)$, then (\ref{e:exist}) holds, and $\sigma(t) = \gamma(t+r-\delta)$ for $ 0\leq t\leq 2\delta$.  Thus
\begin{equation}
L_o(\sigma(t))= \left\{
\begin{array}{cc}
   r-\delta + t& 0\leq t \leq \delta     \\
 r+\delta-t  & \delta\leq t \leq 2\delta  
\end{array}\right.
.
\end{equation}
It would be impossible for any side $\tilde\sigma$ in a corresponding triangle $\triangle\tilde o\tilde p\tilde q$ to satisfy $L_{\tilde o}(\tilde\sigma(t))\geq L_o(\sigma(t))$,   because otherwise $(L_{\tilde o}\circ \tilde \sigma)^\prime(0)=1$ which shows that $\tilde \sigma$ lies along meridians of $\widetilde M$.  But, $\tilde \sigma$ has to join two points on the same line of latitude at distance $r-\delta$ from $\tilde o$. If $\ell = \infty$, this is clearly impossible, while if $\ell$ is finite it is also impossible since $\tilde \sigma$  has to pass through the antipodal vertex to $\tilde o$ and so have length 
$$2( \ell - r + \delta) > 2\delta.$$
Whereas  the length of $\tilde \sigma$ is $2\delta$.

This leads to the following conclusion:

\begin{proposition}\label{p:inj}%3.13
Suppose $(\widetilde M, \tilde o)$ has stronger radial attraction than $(M,o)$.  If for every triangle $\triangle opq$ satisfying Equation (\ref{e:exist}), $$ L_o\circ \sigma (t) \leq  L_{\tilde o} \circ \tilde \sigma(t)$$
 for all $0\leq t \leq d(p,q)$, then the distance from $o$ to the nearest  cut point is  at least $\ell$. In other words, the injectivity radius of $M$ at $o$ is greater than or equal to $\ell$.
\end{proposition}

\section{Injectivity and Convexity Radii}

\begin{definition} We will say that a subset $A$ of a complete Riemannian manifold is \emph{convex} if for all pairs of points $p$ and $q$ in $A$, every minimizing geodesic joining $p$ and $q$ in $M$ is contained in $A$.  We will say $A$ is \emph{strongly convex} if for all pairs of points $p$ and $q$ in $A$, there exists a unique minimizing geodesic joining $p$ and $q$, and that geodesic is contained in $A$. In other words, $A$ is strongly convex if $A$ is convex and every pair of points in $A$ are joined by a unique minimizing geodesic in  $M$.
\end{definition}

Recall that the injectivity radius of $M$ at a point $p\in M$ is defined by
$$inj_M(p)= \sup\{ r>0: \exp_p \mathrm{\ is\ a\  diffeomorphism\ on\ the\ ball\ of\ radius\ }r \mathrm{\ in\ }T_pM \}$$
If $A \subset M$, define
$$inj_M(A) = \inf \{ inj_M(p): p\in A\}.$$
Since the metric balls $B_r(o)=\{p: d(o,p)<r\}$ grow with increasing $r$, the function
$f(r) = inj_M(B_r(o))$ is a decreasing function in $r$. Clearly $f(0)= inj_M(o)$ and $\lim_{r\to\infty}f(r) = inj_M(M)$.
Thus there exists an $r^\ast=r^\ast(o)$ where $inj_M(B_{r^\ast}(o))= 2r^\ast$. Hence
$$\frac 12 inj_M(M) \leq r^\ast(o) \leq \frac 12inj_M(o).$$
Therefore if $r <r^\ast(o)$, we have $inj_M(B_r(o))> 2r$,  which implies that every pair of points $p$ and $q$ in $B_r(o)$ are joined by a unique minimizing geodesic in $M$ since $d(p,q)<d(o,p)+d(o,q) <2r$ by the triangle inequality.
 \begin{definition}
The \emph{convexity radius of $M$ at $p$} is defined to be
$$ conv_M(p)=\sup\{0<r<inj_M(p) :B_r(p) \mathrm{\ is\ strongly\ convex.}\}$$
\end{definition}

For further discussion of the convexity radius see \cite{JD}.

\begin{proposition}%4.3
If $(\widetilde M,\tilde o)$ has stronger radial attraction than $(M,o)$ and the geodesics in $M$ have no bad encounters in $\widetilde M$, then
$$ conv_M(o) \geq \min( conv_{\widetilde M}(\tilde o), r^\ast(o)).$$
\end{proposition}
\begin{proof}
Let $r <\min( conv_{\widetilde M}(\tilde o), r^\ast(o))$, and let $p,q\in B_r(o)$. Since $r < r^\ast(o)$ there exists a unique minimizing geodesic $\sigma$ joining $p$ and $q$. Since $ d(p,q)<2r$, the length of $\sigma$ is less than $2r$. Thus for
every point $\sigma(t)$, either $d(p,\sigma(t))<r$ or $d(q,\sigma(t))<r$ so again by the triangle inequality
$$ d(o,\sigma(t)) <2r <2r^\ast(o) \leq inj_{M}(o).$$
Therefore $\sigma$ does not intersect $C(o)$, and thus $L_o\circ\sigma(t)$ has a 2--sided derivative for all $0<t<d(p,q)$.

Since $r < conv_{\widetilde M}(\tilde o)$, the ball $B_r(\tilde o)$ is strongly convex.  It follows that
$$D_\pi(L_o(p),L_o(q)) =  L_o(p)+L_o(q),$$ and therefore $d(p,q) \leq D_\pi(L_o(p), L_o(q))$ by the triangle inequality.  As $conv_{\widetilde M}(\tilde o)\leq \ell$ we also have $d(o,p)<\ell$ and $d(o,q) < \ell$.  

Applying Theorem \ref{t:TCT}, there exists a corresponding triangle $\triangle \tilde o\tilde p\tilde q$ to $\triangle opq$ such that
$L_o\circ\sigma(t) \leq L_{\tilde o}\circ \tilde \sigma(t)$ for all $0\leq t \leq d(p,q)$.  But the strong convexity of $B_r(\tilde o)$ implies that $L_{\tilde o}\circ\tilde\sigma(t) < r$. Thus $L_o\circ\sigma(t) <r$. We conclude that $B_r(o)$ is strongly convex.
This completes the proof.
\end{proof}

\section{Radial curvature bounded above}\label{s:curvature}

\begin{definition}
Given a pointed Riemannian manifold $(M,o)$ and model surface $(\widetilde M, \tilde o)$, $(M,o)$ has \emph{radial curvature bounded from above} by $(\widetilde M, \tilde o)$ if for every unit speed geodesic $\gamma$ emanating from $o$, the sectional curvature of every 2--plane $X\wedge \gamma^\prime(r)$, where $X\perp \gamma^\prime(r)$, is less than or equal to $\kappa(r)$ for every $0<r<\ell$ where $\kappa(r) = -\frac{y^{\prime\prime}(r)}{y(r)}$ is the Gaussian curvature function of $\widetilde M$.
\end{definition}
 
\subsection{The Hessian Comparison}
\begin{proposition}\label{p:HCT}%5.2
If $(M,o)$ has radial curvature bounded from above by $(\widetilde M,\tilde o)$, then
$$S(X,X) \leq 0$$
for all vector fields $X$ on  $(M\backslash (\{o\} \cup C(o)))  \cap \{p : L_o(p) <\ell\}$ where $S$ is the tensor field defined in  Section \ref{s:TFS}. 
\end{proposition}

\begin{proof}
This follows by applying the Hessian Comparison Theorem \cite[Theorem A]{GW} with $f=L_o$, $(M,o)=(M,o)$ and $(N,p)= (\widetilde M^n,\tilde o^n)$.  Although the hypothesis in  \cite[Theorem A]{GW} is that the manifolds have a pole, the proof only requires that the geodesics be free of conjugate points. This holds in this case on account  of Proposition \ref{p:ccp}. 
\end{proof}

\subsection{Partial Converse to Lemma \ref{l:S}}

\begin{proposition}\label{p:CS}%5.3
Given the pointed Riemannian manifold $(M,o)$ and model space $(\widetilde M,\tilde o)$, if $$S(X,X) \leq 0$$
for all vector fields $X$ on  $(M\backslash (\{o\} \cup C(o)) ) \cap \{p : L_o(p) <\ell\}$ where $S$ is the tensor field defined in  Section \ref{s:TFS} and if $M$ is real analytic, then $(\widetilde M,\tilde o)$ has stronger radial attraction than $(M,o)$.
\end{proposition}

\begin{proof}
Let $\sigma$ and $\tilde\sigma$ be geodesics in $M$ and $\widetilde M$ respectively such that $L_o\circ \sigma(0) = L_{\tilde o}\circ \tilde \sigma(0)<\ell$ and $(L_o\circ \sigma)^\prime_+(0) = (L_{\tilde o}\circ \tilde \sigma)^\prime_+(0)$. 
By Remark \ref{r:<1} we may assume that $|(L_o\circ \sigma)^\prime_+(0)| <1$. 

Set $p=\sigma(0)$
and $\tilde p = \tilde\sigma(0)$, and set $X=\sigma^\prime(0)$ and $\tilde X=\tilde\sigma^\prime(0)$. 
Now there exists a $b>0$ such that $L_o\circ \sigma(t)$ has a 2-sided derivative for all $0<t<b$.  This is true if $p \notin C(o)$
by Remark \ref{r:2sided},  because  there exists a $b>0$ such that $\sigma(t) \notin C(o)$ for $0\leq t \leq b$, and is also true if 
$p\in C(o)$ by Lemma \ref{l:rays} below, because $p$ is a nonconjugate cut point by  Proposition \ref{p:ccp}.

The tensor field $S$ is defined on the set $(M\backslash( \{o\} \cup C(o)) ) \cap \{p : L_o(p) <\ell\}$.  If $p \in  C(o)  \cap \{p : L_o(p) <\ell\}$, one can extend $S$ to a symmetric $(2,0)$ tensor $\overline S_p$ on $T_pM$ by taking the limit of $S$ along a minimizing geodesic joining $o$ to $p$.  Of course the limit depends on which minimizing geodesic one uses.
If, in the notation of Lemma \ref{l:rays}, one uses the geodesic indexed by $i$, then one has
\begin{equation}
\overline S_p =\frac{y^\prime\circ L_o(p)}{y\circ L_o(p)} \left(g_p - (df_i)_p\otimes (df_i)_p\right)  -(\nabla^2f_i)_p
\end{equation} 
where $g$ is the Riemannian metric on $M$.  Since $\overline S_p$ is the limit of negative symmetric tensors, we have
$\overline S_p(X,X) \leq 0$.

Now it is  known that $(L_o\circ\sigma)^{\prime\prime}(0) = \nabla^2L_o(X,X)$ and $(L_{\tilde o}\circ\tilde\sigma)^{\prime\prime}(0)=\nabla^2L_{\tilde o}(\tilde X,\tilde X)$. (See for example \cite[(2.11)]{HI1}.) Since $S(X,X) \leq 0$, it follows that
\begin{equation}\label{e:secondderivative}
(L_{\tilde o}\circ\tilde\sigma)^{\prime\prime}(0)=\nabla^2L_{\tilde o}(\tilde X,\tilde X)\leq\nabla^2L_o(X,X)=(L_o\circ\sigma)^{\prime\prime}(0).
\end{equation}
When $p \in C(o)$, by the proof of Lemma \ref{l:rays} there is an index $i$ such that $L_o\circ \sigma = f_i\circ\sigma$ and Equation (\ref{e:secondderivative}) becomes
\begin{equation}\label{e:cutsecondderivative}
(L_{\tilde o}\circ\tilde\sigma)^{\prime\prime}(0)=\nabla^2L_{\tilde o}(\tilde X,\tilde X)\leq\nabla^2f_i(X,X)=(L_o\circ\sigma)^{\prime\prime}(0)
\end{equation}
because $\overline S_p(X,X) \leq 0$.  If there is strict inequality in (\ref{e:secondderivative}) or (\ref{e:cutsecondderivative}),
calculus implies that there exists an $\epsilon >0$ such that
$$ L_{\tilde o}\circ \tilde\sigma(t) < L_o\circ\sigma(t)$$
for all $0<t<\epsilon$, since $L_{\tilde o}\circ \tilde\sigma(t)$ and  $ L_o\circ\sigma(t)$ agree to first order at $0$, but  
$(L_{\tilde o}\circ\tilde\sigma)^{\prime\prime}(0)<(L_o\circ\sigma)^{\prime\prime}(0)$.  

Next assume that $S(X,X)$ or $\overline S_p(X,X)$ equals  $0$ so that   $(L_{\tilde o}\circ\tilde\sigma)^{\prime\prime}(0)=(L_o\circ\sigma)^{\prime\prime}(0)$.  Fix $r_1$ and $r_2$ such that
$$ 0<r_1< d(o,p) < r_2 <\ell.$$ 

\begin{lemma}\label{l:delta}%5.4
 For all  sufficiently small $\delta>0$ there exists a model surface $(\widetilde M_\delta,\tilde o_\delta)$ 
with metric
$$ds^2_\delta = dr^2 + y_\delta^2(r)d\theta^2$$
for $(r,\theta) \in (0,\ell-\delta)\times[0,2\pi]$
such that
$$ \frac {y_\delta^\prime(r)}{y_\delta(r)} \leq \frac {y^\prime(r)}{y(r)}$$
for all $0<r<\ell-\delta$ and
$$ \frac {y_\delta^\prime(r)}{y_\delta(r)} +\delta = \frac {y^\prime(r)}{y(r)}$$
for $r_1<r<r_2$ where $y(r)$ is the function defining the metric on $\widetilde M$.
Furthermore, as $\delta \rightarrow 0^+$, $\widetilde M_\delta$ converges to $\widetilde M$ uniformly on $(0,R)\times [0,2\pi]$ for any fixed $R<\ell$. 
\end{lemma}

\begin{proof}
We construct $\widetilde M_\delta$ using the method found in \cite[Section 6]{HI1}.  

If $\ell=\infty$, define $y_\delta(r) = m_\delta(r) y(r)$ for $ 0\leq r$ where
$$
m_\delta(r) = \left\{\begin{array}{ccc} \varphi(r) + (1-\varphi)e^{-\delta r}&\mathrm{if} & 0 \leq r \leq r_1\\ e^{-\delta r}&\mathrm{if}& r_1\leq r \end{array}     \right.
$$
and $\varphi: [0,\infty)\rightarrow [0,1]$ is a smooth decreasing function which is $1$ in a neighborhood of $0$ and $0$ for $r\geq r_1$.
 Using \cite[Lemma 6.1]{HI1}, one easily checks that $y_\delta$ has the desired properties. Indeed, by construction,
 $$  \frac {y_\delta^\prime(r)}{y_\delta(r)} +\delta = \frac {y^\prime(r)}{y(r)}$$
for all $r>r_1$, not just for $r_1<r<r_2$.
 
The case $\ell < \infty$ requires a modification in the definition of $m_\delta$ to ensure that $y_\delta(\ell-\delta)=0$ and $y_\delta^\prime(\ell-\delta)=-1$.  For each  $\delta>0$,  there exists an $r_3=r_3(\delta)$ such that  $\frac {\ell-\delta-r}{y(\ell-\delta)} < e^{-\delta r}$ for all $r_3 < r  < \ell-\delta$. Clearly $r_3(\delta)$ approaches $\ell$ as $\delta$ approaches $0$ because ${y(\ell-\delta})$ approaches $0$ as well. Thus $ r_3(\delta) > r_2$ for all sufficiently small $\delta$. Let's say for $\delta <\delta^\ast$. Let $\varphi$ be defined as before, and define $\psi_\delta:[0,\infty) \rightarrow [0,1]$ to be a smooth decreasing function such that $\psi_\delta(r) = 1$ for $0<r< r_3(\delta)$,  is $1$ in a neighborhood of $r_3$ and is $0$ in a neighborhood of $\ell-\delta$.
For every $ \delta < \delta^\ast$, define
$y_\delta(t) = m_\delta(r) y(r)$ for $ 0 \leq r <\ell-\delta$ where
$$
m_\delta(r) = \left\{\begin{array}{ccc} \varphi(r) + (1-\varphi(r))e^{-\delta r}&\mathrm{if} & 0 \leq r \leq r_1\\ \psi_\delta(r)e^{-\delta r}+(1-\psi_\delta(r))\frac {\ell-\delta-r}{y(\ell-\delta)}&\mathrm{if}& r_1\leq r \leq \ell-\delta\end{array}.     \right.
$$
Using \cite[Lemma 6.1]{HI1}, one easily checks that $y_\delta$ has the desired properties.  Indeed, by construction,
 $$  \frac {y_\delta^\prime(r)}{y_\delta(r)} +\delta = \frac {y^\prime(r)}{y(r)}$$
for all $r_1<r<r_3(\delta)$, not just for $r_1<r<r_2$.

By construction in either case, for any $R < \ell$, $m_\delta(r)$ converges to $1$ uniformly for $r \in [0,R]$ as $\delta$ approaches $0$. Therefore
$y_\delta(r)$ converges to $y(r)$ uniformly for $r \in [0,R]$.
 \end{proof}

Define the tensor field
$$S_\delta =  \frac{y_\delta^\prime\circ L_o}{y_\delta\circ L_o} \left(g - dL_o\otimes dL_o\right) - \nabla^2L_o$$
analogous to $S$. By construction, for all non--radial vectors  $Z$ tangent at a point $q$ in $(M\backslash (\{o\} \cup C(o)) ) \cap \{p : L_o(p) <\ell\}$ with $L_o(q) \in [r_1,r_2]$, we have
 $S_\delta(Z,Z) < S(Z,Z)$.
With no loss of generality we can suppose that $b$ has been chosen small enough so that $L_o\circ\sigma(t)$ lies in the interval $[r_1,r_2]$ for all $0\leq t\leq b$. Because $|(L_o\circ \sigma)^\prime_+(0)| <1$, $\sigma^\prime(t)$ is never radial.
Hence, the second derivative $(L_o\circ\sigma)^{\prime\prime}(t_0)$ is strictly greater than that of the geodesic in $\widetilde M_\delta$ that agrees with $\sigma$ to first order at $t_0$. We conclude that the hypothesis that $\widetilde M_\delta$ has stronger radial attraction than $M$ holds at every point along $\sigma$.

Let $\tilde\sigma_\delta$ be the geodesic in $\widetilde M_\delta$ such that $L_{\tilde o_\delta}\circ\tilde\sigma_\delta(0) =L_o\circ\sigma(0)$ and $(L_{\tilde o_\delta}\circ\tilde\sigma_\delta)_+^\prime(0) =(L_o\circ\sigma)_+^\prime(0)$.
Set $\tilde p_\delta = \tilde\sigma_\delta(0)$.  Since $\widetilde M_\delta$ converges to $\widetilde M$, $inj_{\widetilde M_\delta}(\tilde p_\delta)$ converges to $inj_{\widetilde M}(\tilde p)$. Hence for all sufficiently small $\delta$ we have 
$inj_{\widetilde M_\delta}(\tilde p_\delta) \geq \frac 12 inj_{\widetilde M}(\tilde p)$. Since $\delta<\frac 1 2(\ell-r_2)$, we have
$\ell-\delta > \frac 12(\ell+r_2)-L_o(p)$.  Thus for all sufficiently small $\delta$ we have
$$ \min\{ b, \frac 12(\ell+r_2)-L_o(p), \frac 12 inj_{\widetilde M}(\tilde p)\} \leq \min\{b,\ell-\delta, inj_{\widetilde M_\delta}(\tilde p_\delta) \}.$$
Applying Corollary \ref{c:UE}, we have if $\epsilon < \min\{ b, \frac 12(\ell+r_2)-L_o(p), \frac 12 inj_{\widetilde M}(\tilde p)\}$
then $L_o\circ \sigma(t) \geq L_{\tilde o}\circ \tilde\sigma_\delta(t)$ for all $0 \leq t < \epsilon$.  Note that  $\epsilon$ is independent of $\delta$ for sufficiently small $\delta$.  Letting $\delta \rightarrow 0^+$, $\tilde\sigma_\delta(t)$ converges to
$\tilde\sigma(t)$. Therefore $L_o\circ \sigma(t) \geq L_{\tilde o}\circ \tilde\sigma(t)$ for all $0 \leq t < \epsilon$. 
This completes the proof.
\end{proof}

Combining Propositions \ref{p:HCT} and \ref{p:CS} gives the following result:

\begin{corollary}\label{c:CS}%5.5
If $(M,o)$ has radial curvature bounded from above by $(\widetilde M,\tilde o)$, and $M$ is real analytic, then
 $(\widetilde M,\tilde o)$ has stronger radial attraction than $(M,o)$.  
\end{corollary}

It is an open question whether real analyticity can be removed from the hypothesis in Corollary \ref{c:CS}.

\subsection{Example} In general, if $(\widetilde M,\tilde o)$ has stronger radial attraction than $(M,o)$, then  the radial curvature of $(\widetilde M, \tilde o)$ does not necessarily  bound that  of  $(M,o)$ from above.  Given a rotationally symmetric surface, Example 1 in \cite[Section 6]{HI1} describes how to construct a model which has stronger radial attraction than that of the given surface, yet  does not bound its radial curvature from above.   
 
\section{Volume and Eigenvalue Comparison}

Under the assumption of stronger radial attraction we prove lower bounds on the volume and on the first eigenvalue of the Laplacian for the Dirichlet problem of geodesic balls about $o$.  

\begin{theorem}\label{t:vec}%6.1
Suppose that $(\widetilde M, \tilde o)$ has stronger radial attraction than $(M,o)$ where $n = \dim(M)$. Let $ \rho <\min\{ inj_M(o), \ell\}$.
Let $B_\rho(o)$ be the geodesic ball of radius $\rho$ in $M$ and $B_\rho(\tilde o^n)$ be the geodesic ball in the $n$--dimensional model $(\widetilde M^n, \tilde o^n)$ corresponding to $(\widetilde M, \tilde o)$. Then
$$ Vol(B _\rho(o)) \geq Vol(B _\rho(\tilde o^n))$$
where $Vol$ is the $n$--dimensional volume of the respective geodesic balls,
and
$$\lambda_1(B _\rho(o)) \geq \lambda_1(B _\rho(\tilde o^n))$$
where $\lambda_1$ is the first eigenvalue of the Laplacian for the Dirichlet problem on the respective geodesic balls.
\end{theorem}

The proof is contained in a sequence of lemmas and remarks.

\begin{lemma}\label{l:volform}%6.2
There is a $C^\infty$ function $\alpha: (0,\rho)\times S^{n-1} \to (0,\infty)$ such that  the Riemannian volume form $ \nu_M$  of $M$  takes the form
$$ \nu_M =\alpha(r, X) dr\wedge\nu_{S^{n-1}}$$
in geodesic spherical coordinates about $o$ in $B_\rho(o)$,
where $\nu_{S^{n-1}}$ denotes the $(n-1)$--dimensional volume form of the unit $(n-1)$--sphere $S^{n-1}$.
Moreover $ \lim_{r\to 0^+}\frac {\alpha(r,X)}{r^{n-1}}=1$.
\end{lemma}

\begin{proof}
The spherical coordinates are obtained via the diffeomorphism 
$$E: (0,\rho)\times S^{n-1} \rightarrow B_\rho(o)\backslash\{o\}$$ 
defined by $E(r,X)=\exp_o(rX)$.  Since $ dr\wedge \nu_{S^{n-1}}$ is a nonvanishing $n$--form there is a function $\alpha$ such that the pull--back
$E^\ast(\nu_M)= \alpha(r,X) dr\wedge \nu_{S^{n-1}}$. Moreover,  $r^{n-1}dr\wedge \omega_{S^{n-1}}$ is the volume element in spherical coordinates in the tangent space $T_oM$, and hence
$$\frac{\alpha(r,X)}{r^{n-1}} r^{n-1} dr\wedge \omega_{S^{n-1}} = \exp_o^\ast(\nu_M).$$
Finally, because $\exp_o^\ast(\nu_M)$ at the origin is equal to the volume element of $T_oM$,
$$\lim_{r\to 0^+}\frac {\alpha(r,X)}{r^{n-1}}=1.$$
\end{proof}

\begin{remark}\label{r:ModelVolForm}%6.3
In spherical coordinates about $\tilde o^n$ in the $n$--dimensional model corresponding to $(\widetilde M, \tilde o)$, the metric takes the form 
$$ds^2 = dr^2 + y^2(r) d\theta^2_{n-1},$$
where $d\theta^2_{n-1}$ is the standard metric on the unit $(n-1)$--sphere. Thus
$$\nu_{\widetilde M^n}= y^{n-1}(r) dr\wedge \nu_{S^{n-1}}.$$
\end{remark}

\begin{lemma} \label{l:VolElement}%6.4
For all $(r,X)\in (0,\rho)\times S^{n-1}$,
$$ \frac {\partial_r \alpha(r,X)}{\alpha(r,X)} \geq (n-1)\frac{y^\prime(r)}{y(r)}.$$
Consequently 
$$ \alpha(r,X) \geq y^{n-1}(r).$$
\end{lemma}
\begin{proof}
On  one hand the Laplacian $\Delta L_o$  is equal to the trace of the Hessian $\nabla^2 L_o$.  Thus taking the trace of the tensor field $S$, we have by definition of $S$ and Lemma \ref{l:S},
\begin{equation}\label{e:LPC}
 \Delta L_o= tr(\nabla^2L_o) \geq (n-1)\frac{y^\prime(r)}{y(r)}.
 \end{equation}
On the other hand, $\Delta L_o$ is the divergence of the gradient of $L_o$. Now $grad(L_o) =\xi$, the unit radial vector field.
Thus 
$$ (\Delta L_o) \nu_M = d( i_\xi \nu_M) = d( \alpha \nu_{S^{n-1}}) = (\partial_r\alpha )dr\wedge \nu_{S^{n-1}}= \frac {\partial_r \alpha}{\alpha}\nu_M.$$
This shows that $\Delta L_o=\frac {\partial_r \alpha(r,X)}{\alpha(r,X)}$. 
Plugging this into (\ref{e:LPC}) gives the first inequality.
Notice that the terms in the inequality (\ref{e:LPC}) are logarithmic derivatives. Hence
$$\partial_r \log(\alpha) \geq \frac d{dr}\log (y^{n-1}).$$
If $0<r_0 < r<\rho$, integrating from $ r_0$ to $r$ gives
$$\log(\alpha(r,X))-\log(\alpha(r_0,X)) \geq \log(y^{n-1}(r))-\log(y^{n-1}(r_0)).$$
Using properties of the logarithm and some algebra we easily obtain
$$ \frac {\alpha(r,X)}{y^{n-1}(r)} \geq \frac{\alpha(r_0,X)}{y^{n-1}(r_0)}.$$
Now if we let $r_0$ approach $0$, the righthand side approaches $1$ by Lemma \ref{l:volform}.
Therefore,  for all $(r,X)$ with $ r< \rho$, we obtain
$$ \frac {\alpha(r,X)}{y^{n-1}(r)} \geq 1$$
which completes the proof.

\end{proof}

\begin{remark}%6.5
Therefore, applying Lemma \ref{l:VolElement} and Remark \ref{r:ModelVolForm},
\begin{eqnarray*}
Vol(B_\rho(o)) &=& \int_0^\rho\int_{X\in S^{n-1}} \alpha(r,X) dr\wedge \nu_{S^{n-1}}\cr &\geq & \int_0^\rho\int_{X\in S^{n-1}}y^{n-1}(r) dr\wedge \nu_{S^{n-1}}\cr&=& Vol(B_\rho(\tilde o)).
\end{eqnarray*}

\end{remark}

\begin{remark}\label{r:1stEigen}%\6.6
Because of the rotational symmetry of the model space $\widetilde M^n$ about $\tilde o^n$, the eigenfunction of the first eigenvalue $\lambda = \lambda_1(B_\rho(\tilde o^n))$ for the Dirichlet problem on $B_\rho(\tilde o^n)$ are multiples of a function $\varphi$ that depends only on $r$ such that
$\Delta \varphi + \lambda\varphi =0$, $\varphi(\rho)=0$. Moreover it is known that $\lambda>0$ and that $\varphi(r)>0$ for $0<r<\rho$.
\end{remark}

\begin{lemma} \label{l:1stEigen}%6.7
For $0<r<\rho$,
$\varphi^\prime(r)<0$.
\end{lemma}
\begin{proof}
Clearly $grad (\varphi (r) )= \varphi^\prime(r) \xi$. Hence
\begin{equation}\label{e:1stEigen}
(\Delta\varphi(r)) \nu_{\widetilde M^n} = d(\varphi^\prime(r)y^{n-1}(r) dr\wedge\nu_{S^{n-1}}) = \frac{\partial_r(\varphi^\prime(r)y^{n-1}(r))}{y^{n-1}(r)}\nu_{\widetilde M^n}.
\end{equation}
Thus by Remark \ref{r:1stEigen}
\begin{equation}\label{e:Eigen}
\partial_r(\varphi^\prime(r)y^{n-1}(r))+\lambda \varphi(r) {y^{n-1}(r)}=0.
\end{equation}
Therefore, integrating (\ref{e:Eigen}) from $0$ to $r$ we have
\begin{equation*}
\varphi^\prime(r)y^{n-1}(r) = -\lambda\int_0^r \varphi(r) {y^{n-1}(r)} dr < 0.
\end{equation*}
Because $\lambda>0$, $y(0)=0$ and $y(r)$ and $\varphi(r)$ are both positive on the open interval $(0, \rho)$.
It follows that $\varphi^\prime(r) < 0$.
\end{proof}

\begin{remark}\label{r:DeltaPhi}%6.8
After expanding the partial derivative in Equation  (\ref{e:1stEigen}) and rearranging the terms, we have the formula:
$$ \Delta \varphi(r) = \varphi^{\prime\prime}(r)+(n-1)\frac{y^\prime(r)}{y(r)}\varphi^\prime(r).$$
\end{remark}

\begin{lemma}\label{l:ratio}%6.9
Define the function $F$ on $B_\rho(o)$ by $F(r,X)= \varphi(r)$. Then
$$\frac {\Delta F}{F}(r,X) \leq \lambda$$
where $\lambda = \lambda_1(B_\rho(\tilde o))$ and $\varphi$ is as in Remark \ref{r:1stEigen}.
\end{lemma}
\begin{proof}
\begin{eqnarray*}
\frac {\Delta F}{F}(r,X)&=& \frac{\partial_r(\alpha(r,X)\partial_rF(r,X))}{\alpha(r,X)F(r,X)}\\
&=& \frac 1 {F(r,X)}\left(\partial_r^2 F(r,X) + \frac{\partial_r\alpha(r,X)}{\alpha(r,X)}\partial_r F(r,X)\right)\\
&\leq& \frac 1 {\varphi(r)}\left(\varphi^{\prime\prime}(r)+(n-1)\frac{y^\prime(r)}{y(r)}\varphi^\prime(r)\right)\\
&=&\frac 1 {\varphi(r)}(\Delta \varphi(r,X))\\
&=&-\lambda.
\end{eqnarray*}
Note that $\varphi^\prime(r) = \partial_r F(r,X) <0$ by Lemma \ref{l:1stEigen}.
We also used  Lemma \ref{l:VolElement} and Remarks \ref{r:1stEigen} and \ref{r:DeltaPhi}.
\end{proof}

\begin{remark}%6.10
By \cite[Lemma 3.3]{IC} 
$$-\lambda_1 (B_\rho(o)) \leq \sup_{B_\rho(o)}\frac {\Delta F}{F}.$$
Combining this with Lemma \ref{l:ratio} gives
$$-\lambda_1 (B_\rho(o)) \leq \sup_{B_\rho(o)}\frac {\Delta F}{F} \leq -\lambda_1(B_\rho(\tilde o)).$$
This completes the proof of Theorem \ref{t:vec}.
\end{remark}

\section{Triangle pinching theorems}

The well known Sphere Theorem \cite[Theorem 6.1]{CE} asserts that a simply connected, complete Riemannian manifold whose sectional curvatures lie in the half open interval
$(\frac 14, 1]$ is homeomorphic to a sphere.  It has an analog in the context of weaker and stronger radial attraction.

\begin{theorem}\label{t:TPT} %7.1
Let $(M,o)$ be a complete pointed Riemannian manifold. Let $(\widetilde M, \tilde o)$ and $(\widehat M,\hat o)$ be two model surfaces with
$\max\{ d(\tilde o,\tilde p) : \tilde p \in \widetilde M \} = \tilde \ell$ and $\max\{ d(\hat o,\hat p): \hat p \in \widehat M\} = \hat\ell$ both finite.
Assume that:
\begin{enumerate}
\item For every geodesic triangle $\triangle opq$ in $M$, there exists a corresponding triangle $\triangle \hat o\hat p\hat q$ in $\widehat M$
with same side lengths and satisfying $L_o(\sigma(t)) \geq L_{\hat o}( \hat\sigma(t))$ for $0\leq t \leq d(p,q)$,
\item For every geodesic  triangle $\triangle opq$ in $M$ for which there exists a corresponding triangle $\triangle \tilde o\tilde p\tilde q$ in $\widetilde M$ with the same side lengths, we have $L_o(\sigma(t)) \leq L_{\tilde o}(\tilde\sigma(t))$ for  $0\leq t \leq d(p,q)$,
\item And $\tilde \ell > \hat \ell - conv_{\widehat M}(\hat o^\prime)$ where $\hat o^\prime$ is the opposite vertex to $\hat o$.
\end{enumerate}
Then $M$ is homeomorphic to a sphere.
\end{theorem}

\begin{remark}
By \cite[Theorem 1.3]{HI2}, condition (1) is equivalent to $(\widehat M,\hat o)$ having weaker radial attraction than $(M,o)$ and none of the geodesics in $M$ has a bad encounter with cut loci in $\widehat M$ for weaker radial attraction.  By Remark \ref{r:thin}, condition (2) implies that $(\widetilde M,\tilde o)$ has stronger radial attraction than $(M,o)$ but is not a consequence of stronger radial attraction on account of Example \ref{x:cut}. Condition (3) replaces the curvature pinching condition in the Sphere Theorem. For example, $\widehat M$ might be the  2--sphere of constant curvature 1 and $\widetilde M$ a constant curvature 2-sphere with curvature a little greater than $\frac 14$.
\end{remark}

\begin{proof}
The idea of the proof is to show that the distance function $L_o : M \rightarrow \mathbf{R}$ has a unique maximum and no other critical points.
It then follows that $M$ is homeomorphic to a sphere by \cite{MG}.

By Proposition \ref{p:inj}, condition (2) implies that $inj_M(o) \geq \tilde \ell > \hat \ell - conv_{\widehat M}(\hat o^\prime)$. Therefore the open ball 
$\{ p\in M : d(o,p)<\tilde \ell\}$ contains no critical points of $L_o$ and the set $\{ x \in M: d(o,x) \geq \tilde \ell\}$ is nonempty.
Thus if $R = \max \{L_o(x) : x \in M\} $, then $R \geq \tilde \ell > \hat \ell - conv_{\widehat M}(\hat o^\prime)$.  Note also $\hat \ell \geq R$ by \cite[Corollary 3.14]{HI2}.
By the definition of convexity radius, the ball $\{ \hat x \in\widehat M :  L_{\hat o}(\hat x) \geq r\}$ centered at $\hat o^\prime$ is strictly convex if $ r>\hat\ell - conv_{\widehat M}(\hat o^\prime)$.

We first argue that the maximum of $L_o$ is attained at a unique point of $M$.  Assume it is attained at two points $x_1$ and $x_2$, then consider the triangle $\triangle\hat o \hat x_1 \hat x_2$  corresponding to $\triangle ox_1x_2$ by hypothesis (1).  
By the strict convexity of the ball $\{ x \in \widehat M : d(\hat o, x) \geq R \}$ and because $L_o(\sigma(t)) \geq L_{\hat o}( \hat\sigma(t))$ for $0\leq t \leq d(x_1,x_2)$, we obtain the contradiction
 $$ R < d(\hat o, \hat\sigma(t)) \leq d(o, \sigma(t)) \leq R$$
 for $ 0 < t < d(x_1,x_2)$.  Therefore the maximum of $L_o$ is attained at a unique point $p \in M$.

Next we show that there are no critical points of $L_o$ in the set 
 $$\{ x \in M : d( o,x) \geq  \hat\ell - conv_{\widehat M}(\hat o^\prime)\}$$ other than $p$.  If there were, let $q$ be another critical point. Let $\sigma$ be a minimizing geodesic joining $p$ to $q$.  Since $p$ and $q$ are critical points of $L_o$, we may pick minimizing geodesics $\tau$ from $o$ to $ p$ and $\gamma$ from $o$ to $q$ so that $\measuredangle opq \leq \frac \pi 2$ and
 $\measuredangle oqp \leq \frac \pi 2$.
 Consider the geodesic triangle $\triangle opq$ and the corresponding triangle $\triangle \hat o\hat p\hat q$ obtained  by hypothesis (1).
 Because $ \hat\ell-conv_{\widehat M}(\hat o^\prime) \leq d( \hat o, \hat q) < d(\hat o, \hat p)$,  the geodesic $\hat \sigma$ is contained in the convex disk $\{ \hat x \in \widehat M : d(\hat o,\hat x) \geq d(\hat o, \hat q)\}$.  Thus its tangent vector $-\hat\sigma^\prime(d(p,q))$ at $\hat q$ points into the disk. On the other hand
 $\hat\gamma$ is a segment of the meridian starting at $\hat o$. Thus its tangent vector  $-\hat\gamma^\prime(d(o,q))$ at $\hat q$  points out of the disk and is in fact perpendicular to the boundary of the disk.  Therefore
  $\measuredangle \hat o\hat q\hat p > \frac \pi 2$ which contradicts the top angle comparison $\measuredangle \hat o\hat q \hat p \leq \measuredangle oqp$ in the generalized Toponogov triangle comparison theorem \cite[Theorem 1.3]{HI2}.  Therefore there are no critical points of $L_o$ in the set $\{ x \in  M : d( o,x) \geq\hat \ell - convex_{\widehat M}(\hat o^\prime)\}$ other than $p$.
  
  This completes the proof because by (3)
  $$ M =  \{ x\in M : d(o,x)<\tilde \ell\}\cup\{ x \in  M : d(o,x) \geq \hat\ell - convex_{\widehat M}(\hat o^\prime)\}.$$
  
 \end{proof}
 
The Minimal Diameter Theorem \cite[Theorem 6.6]{CE}  asserts that a simply connected, complete Riemannian manifold whose sectional curvatures lie in the closed interval
$[\frac 14, 1]$ is either homeomorphic to a sphere or isometric to a symmetric space.  As a result of the proof of Theorem \ref{t:TPT} we obtain:

\begin{corollary}\label{c:TPT}%7.3
Let $(M,o)$ be a complete pointed Riemannian manifold with $R= \max\{d(o,p):p\in M\}$. Let $(\widetilde M, \tilde o)$ and $(\widehat M,\hat o)$ be two model surfaces with
$\max\{ d(\tilde o,\tilde p) : \tilde p \in \widetilde M \} = \tilde \ell$ and $\max\{ d(\hat o,\hat p): \hat p \in \widehat M\} = \hat\ell$ both finite.
Assume that:
\begin{enumerate}
\item For every geodesic triangle $\triangle opq$ in $M$, there exists a corresponding triangle $\triangle \hat o\hat p\hat q$ in $\widehat M$
with same side lengths and satisfying $L_o(\sigma(t)) \geq L_{\hat o}( \hat\sigma(t))$ for $0\leq t \leq d(p,q)$,
\item For every geodesic  triangle $\triangle opq$ in $M$ for which there exists a corresponding triangle $\triangle \tilde o\tilde p\tilde q$ in $\widetilde M$ with the same side lengths, we have $L_o(\sigma(t)) \leq L_{\tilde o}(\tilde\sigma(t))$ for  $0\leq t \leq d(p,q)$,
\item And $\tilde \ell = \hat \ell - conv_{\widehat M}(\hat o^\prime)$ where $\hat o^\prime$ is the opposite vertex to $\hat o$.
\end{enumerate}
Then:
\begin{enumerate}
\item[(4)] if $R>\tilde \ell$,  $M$ is homeomorphic to a sphere.
\item[(5)] if $R=\tilde \ell$,  $M$ is an Allamigeon--Warner manifold at $o$.
\end{enumerate}
\end{corollary}

\begin{proof}
If $R > \tilde\ell$, the proof of Theorem \ref{t:TPT} shows that $L_o$ has no critical points other than at the unique maximum, and hence $M$ is homeomorphic to a sphere.

Since by Proposition \ref{p:inj}, (2) implies $inj_M(o) \geq \tilde \ell$, if  $R=\tilde \ell$, then the cut locus of $o$ in $T_oM$ is the sphere of radius $R$ centered at the origin. Thus in the terminology of \cite[Definition 5.22]{AB}, $M$ has a spherical cut locus at $o$. Therefore $M$ is an Allamigeon--Warner manifold at $o$ by \cite[Therorem 5.43]{AB}.
\end{proof}

\begin{remark}
 Allamigeon--Warner manifolds at a point can be expressed as a topological union of a closed disk and a disk bundle over a submanifold. 
 Simply connected ones have the homology type of a compact rank one symmetric space. See \cite{AA,AB,FW}.  
\end{remark}

\section{A property of short geodesic rays emanating from a nonconjugate cut point}

The following result is needed to prove Proposition \ref{p:CS}.

\begin{lemma}\label{l:rays}%8.1
Let $M$ be a real analytic Riemannian manifold, and let $o\in M$.  If $p\in C(o)$ is a nonconjugate cut point , then for every unit speed geodesic $\sigma$ emanating from $p$, there exists a $b >0$ such that $L_o\circ\sigma$ has a 2--sided derivative $(L_p\circ\sigma)^\prime(t)$ for all $0<t<b$. 
\end{lemma}
\begin{proof}
The proof relies on  the structure of the cut locus in a neighborhood of a nonconjugate cut point given by Ozols \cite{VO}. 
 Let us briefly describe this result. Let $p \in C(o)$ be a nonconjugate cut point. Then there are finitely many minimizing geodesics joining $o$ to $p$.  Thus there exist vectors $Y_i\in T_oM$, $i= 1,\dots, N$, tangent to these geodesics of length $|Y_i| =d(o,p)$ such that $\exp_o(Y_i)=p$ for all $i=1,\dots,N$. Because $p$ is not conjugate along any of these geodesics, the differential of $\exp_o$ is nonsingular at every $Y_i$. Applying the Inverse Function Theorem, there exist a neighborhood $U$ of $p$ in $M$ and  neighborhoods $V_i$ of $Y_i$ in $T_oM$, such that
 $$ \exp_o|V_i : V_i \rightarrow U$$
 is a diffeomorphism for all $i$.  When $M$ is real analytic, these diffeomorphisms are real analytic.
 Thus we can define real-valued functions
 $$f_i : U \rightarrow \mathbf{R}$$
 by $f_i(q) = |(\exp_o|V_i)^{-1}(q)|$ for $i=1,\dots,N$. When $M$ is real analytic, so are these functions. 
 
According to \cite{VO}, if we set
$K_{ij} = \{q\in U : f_i(q)=f_j(q)\}$, $H_{ij} =\{q \in U : f_i(q) \geq f_j(q)\}$, and $C_{ij} = K_{ij} \cap \bigcap_{k=1}^N H_{ki}$,
then
$$ C(o) \cap U = \bigcup_{i<j} C_{ij},$$
because $q \in C(o)\cap U$ if and only if there are at least two minimizing geodesics joining $o$ to $q$, hence if and only if there exists a pair $i<j$ such that $f_i(q)=f_j(q)$ and for all $k$,  $f_k(q) \geq f_i(q)$. Moreover, if $U$ has been chosen small enough, the gradients of $f_i$ and $f_j$ for $i\neq j$ are independent throughout $U$. Thus we can assume that the $K_{ij}$ for $i\neq j$ are hypersurfaces whose tangent spaces at $q \in K_{ij}$ are given by 
$$T_qK_{ij} = ker( (df_i)_q -(df_j)_q).$$

Finally we note that the distance function $L_o$ from $o$ is clearly given by
\begin{equation}\label{e:Lo}
L_o = \min\{ f_1,\dots,f_N\}.
\end{equation}

Now let $X \in T_pM$ be a unit tangent vector, and let $\sigma(t) = \exp_p(tX)$ for $0<t < r$ be the geodesic ray segment  contained in $U$ emanating from $p$ with initial tangent vector $\sigma^\prime(0) = X$. We must prove that there exists a $b >0$ such that $L_o\circ\sigma$ has a 2--sided derivative $(L_o\circ\sigma)^\prime(t)$ for all $0<t<b$. This is accomplished by considering three cases.

(1) Suppose that $X$ is not tangent to any of the hypersurfaces $K_{ij}$. Then the geodesic $\sigma$ is transverse to $K_{ij}$ at $p$ for any $i$ and $j$. Hence for sufficiently small $t$, $\sigma(t) \notin K_{ij}$. Therefore, as  there are only finitely many such hypersurfaces, there exists a $b>0$ such that $\sigma(t) \notin \cup_{i<j} K_{ij}$ for all $0<t<b$. Therefore
$\sigma(t) \notin C(o)\cap U$ for all $0<t<b$ which implies that $L_o\circ\sigma$ has a 2--sided derivative $(L_o\circ\sigma )^\prime(t)$ for all $0<t<b$. 

Next suppose that $X \in T_pK_{ij}$ for some $i < j$ and set $$a = \inf\{0<t<r : \sigma(t) \in C(o) \}.$$ 

(2) If $a>0$, set $b=a$. Then $\sigma(t) \notin C(o)\cap U$ for all $0<t<b$ which again implies that $L_o\circ\sigma$ has a 2--sided derivative $(L_o\circ\sigma )^\prime(t)$ for all $0<t<b$. 
 
 (3) If $a=0$, then since there are only finitely many $C_{ij}$ and infinitely many arbitrarily small $t$ with $\sigma(t) \in C(o)$, there exists at least one pair $i<j$ and an infinite decreasing sequence
 $$ t_1 > t_2 >t_3 > \cdots> t_n > \cdots$$
 converging to $0$ such that $\sigma(t_n) \in C_{ij}$ for all $n$. Then $ f_i(\sigma(t_n))-f_j(\sigma(t_n))=0$ for all $n$.
 Therefore, by real analyticity of $f_i\circ\sigma-f_j\circ\sigma$, it follows that $ f_i(\sigma(t))-f_j(\sigma(t))=0$ for all $t$.
  Suppose $\bar i <\bar j$ is another such pair with 
 $\sigma(\bar t_n) \in C_{\bar i\bar j}$ for an infinite decreasing sequence of $\bar t_k$ converging to $0$.
In the same way we conclude that $f_{\bar i}\circ\sigma(t )=f_{\bar j}\circ\sigma(t)$ for all $t$. If $\bar i$ or $\bar j$ is one of $ i$ or $j$, then  $f_i\circ\sigma(t) = f_j\circ\sigma(t) = f_{\bar i}\circ\sigma(t )=f_{\bar j}\circ\sigma(t)$ for all $t$.  Suppose $\{i,j\}$ and $\{\bar i,\bar j\}$ are disjoint. By taking subsequences we can assume that 
$$ t_1 > \bar t_1 > t_2 > \bar t_2 > \cdots>t_n>\bar t_n>t_{n+1}>\cdots.$$
Thus, since $\sigma(t_n) \in C_{ij}\subset H_{\bar i i}$,
$$ f_i\circ \sigma(t_n)  \leq f_{\bar i} \circ\sigma(t_n)$$
and since $\sigma(t_n) \in C_{\bar i\bar j}\subset H_{i\bar i}$,
$$ f_{\bar i}\circ \sigma(\bar t_n) \leq f_{i}\circ\sigma(\bar t_n).$$
Thus $f_i\circ\sigma(t_n)-f_{\bar i}\circ\sigma(t_n)\leq 0$ and 
$f_i\circ\sigma(\bar t_n)-f_{\bar i}\circ\sigma(\bar t_n)\geq 0$ displays alternating signs.
By the Intermediate Value Theorem, there exists an infinite sequence $t^*_n$ with $t_n \leq t^*_n \leq \bar t_n$
with $$f_i\circ\sigma(t^*_n) -f_{\bar i}\circ\sigma(t^*_n) = 0.$$  Under the assumption that $M$ is real analytic we obtain
$f_i\circ\sigma(t) =f_{\bar i}\circ\sigma(t)$ for all $t$.

Let $I^\ast$ be the set of all indices $i$ for which there exists a $j$ such that $\sigma(t) \in C_{ij}$ (or $C_{ji}$  if $j<i$) for a sequence of arbitrarily small $t$.
The previous paragraph shows that if $i,j\in I^\ast$, then $f_i\circ\sigma(t) =f_j\circ\sigma(t)$ for all $t$. Next we will prove that if $i\in I^\ast$ and $\bar j\notin I^\ast$, then $f_i\circ\sigma(t) > f_{\bar j}\circ\sigma(t)$ for all sufficiently small $t$.
First observe that there is a  $j \in I^\ast$ and decreasing sequence $t_n$ converging to $0$ such that
$$\sigma(t_n) \in C_{ij} \subset \cap_{k=1}^N H_{k i}$$
for all $n$. In particular $f_{\bar j}(\sigma(t_n)) \geq f_i(\sigma(t_n))$ for all $n$. Now if there were arbitrarily small $t$ with
$f_i\circ\sigma(t) \geq f_{\bar j}\circ\sigma(t)$, we could apply the Intermediate Value Theorem and real analyticity as in the previous paragraph and conclude that $f_{\bar j}\circ\sigma(t) = f_i\circ\sigma(t)$ for all $t$. In particular 
$f_{\bar j}\circ\sigma(t_n) = f_i\circ\sigma(t_n) \leq f_k\circ\sigma(t_n)$ for all $t$ and all $k$. Thus $\sigma(t) \in C_{i\bar j}$
for arbitrarily small $t$.
It then follows that $\bar j \in I^\ast$ which is a contradiction.

Thus by Equation (\ref{e:Lo}), there exists a $b>0$ such that  $L_o\circ \sigma(t) = f_i\circ\sigma(t)$ for any $i \in I^\ast$ and $0\leq t <b$.
Therefore, for $0<t<b$, the 2-sided derivative of $L_o\circ\sigma$ at $t$ exists and is given by
$(L_o\circ \sigma)^\prime(t) = (f_i\circ\sigma)^\prime(t)$
for any $i \in I^\ast$ because $f_i$ is differentiable. This completes the proof of case (3).
\end{proof}

\begin{remark}
The analysis of cases (1) and (2) is valid only assuming $M$ is $C^\infty$.  Case (3) is the only place that utilizes the real analyticity hypothesis. Since almost all directions at $p$ are not tangent to any of the $K_{ij}$, case (1) shows that the conclusion of Lemma \ref{l:rays} holds for almost all directions emanating from $p$ when $M$ is only $C^\infty$. 
\end{remark}

\subsection{Examples} The cut loci of ellipsoids of revolution  \cite{ST} provide examples of cases (2) and (3) in the proof of Lemma \ref{l:rays}.
For a nonvertex point in an oblate ellipsoid, the cut locus is an arc  contained in the line of latitude centered about the antipodal point.  The points in the interior of this arc are nonconjugate cut points, and a short ray initially tangent to this  
 arc only meets the cut locus at the initial point as in case (2).  For a nonvertex point in a prolate ellipsoid, the cut locus is an arc in the opposite meridian.  The points in the interior of this arc are nonconjugate cut points, and a short ray initially tangent to this  arc  remains in the cut locus as in case (3).


\begin{thebibliography}{WWW}

\bibitem{AA} A. Allamigeon. Properti\' et\' es globales des espaces de Riemann harmoniques. Ann. Inst. Fourier 15 (1965) 91--132.

\bibitem{AB} A. Besse. Manifolds all of whose Geodesics are Closed. Springer Verlag, Berlin, Heidelberg, New York, 1978.

\bibitem{IC} I. Chavel. Riemannian Geometry: A Modern Introduction. Cambridge University Press, Cambridge, 1993

\bibitem{CE} J. Cheeger and D. Ebin. Comparison Theorems in Riemannian Geometry. North--Holland, Amsterdam,1975.

\bibitem{JD} J. Dibble. The convexity radius of a riemannian manifold. Asian J. Math 21 (2017) 169--174.


\bibitem{GW} R. Greene and H. Wu. Function Theory on manifolds which possess a pole,  Lecture Notes in Mathematics 699. Springer-Verlag, Berlin-Heidelberg-New York, 1979.

\bibitem{MG} M. Gromov. Curvature, diameter, and Betti numbers. Comment. Math. Helv. 56 (1981), 179--195.

\bibitem{HI1} J. Hebda and Y. Ikeda. Replacing the Lower Curvature Bound in Toponogov's Comparison Theorem by a Weaker Hypothesis. Tohoku Math. J.  69 (2017), 305--320.

\bibitem{HI2} J. Hebda and Y. Ikeda. Necessary and Sufficient Conditions for a Triangle Comparison Theorem. Tohoku Math. J. 74 (2022) 1--36.

\bibitem{HI3} J. Hebda and Y. Ikeda. Generalized Maximal Diameter Theorems. (preprint)

\bibitem{IMS} Y. Itokawa, Y. Machigashira, K. Shiohama. Generalized Toponogov's Theorem for manifolds with radial curvature bounded below. Contemporary Mathematics  332 (2003) 121--130.

\bibitem{ISU} N. Innami, K. Shiohama, and Y. Uneme. The Alexandrov--Toponogov Comparison Theorem for Radial Curvature. Nihonkai Math. J. 24 (2013),  57--91.

\bibitem{VO} V. Ozols. Cut loci in Riemannian manifolds. Tohoku Math. J. 26 (1974), 219--227.


\bibitem{KT} K. Kondo and M. Tanaka. Toponogov comparison theorem for open triangles. Tohoku Math. J. 63 (2011) 363--396.

\bibitem{SST} K. Shiohama, T. Shioya, and M. Tanaka. The geometry of total curvature on complete open surfaces,
Cambridge Tracks in Math. 159. Cambridge University Press, Cambridge, 2003.

\bibitem{ST} R. Sinclair and M. Tanaka.  The cut locus of a two--sphere of revolution and Toponogov's comparison theorem. Tohoku Math. J. 59 (2007) 379--399.

\bibitem {TAS} M. Tanaka, T. Akamatsu, R, Sinclair, M. Yamaguchi. Generalized von Mangoldt surfaces of revolution and asymmetric two-spheres of revolution with simple cut locus structure. 
https://doi.org/10.48550/arXiv.2202.00853.

\bibitem{FW} F. Warner. Conjugate loci of constant order. Ann. Math. 86 (1967) 192--212.
\end{thebibliography}
\end{document}